\documentclass[11pt,reqno]{amsart}
\usepackage{amsmath,amssymb,amsthm,mathtools,bm,array,mathrsfs}
\usepackage{needspace}
\usepackage{microtype}
\usepackage{graphicx}
\usepackage{float}
\usepackage{hyperref}
\usepackage{cite}
\hypersetup{
 hidelinks,
 pdftitle={An Unstable Positive Equilibrium and Hopf Bifurcation in a Predator--Prey System with Spatial Diffusion},
 pdfauthor={Yuan Lou and Ming-Zhen Xin}
}

\allowdisplaybreaks
\renewcommand{\baselinestretch}{1.20}
\numberwithin{equation}{section}

\newtheorem{theorem}{Theorem}[section]
\newtheorem{lemma}[theorem]{Lemma}
\theoremstyle{remark}

\newcommand{\dd}{\mathop{}\!\mathrm d}
\newcommand{\R}{\mathbb R}
\newcommand{\C}{\mathbb C}
\newcommand{\1}{\mathbf 1}
\newcommand{\cL}{\mathcal L}
\newcommand{\cH}{\mathcal H}
\newcommand{\cX}{\mathcal X}
\newcommand{\cP}{\mathcal P}
\newcommand{\cQ}{\mathcal Q}
\newcommand{\spb}{s}
\renewcommand{\Im}{\operatorname{Im}}

\title[ Hopf bifurcation in  a  predator-prey model]{\textbf{Hopf bifurcation in  a  predator-prey model with diffusion and spatial heterogeneity}}
\author[Lou  and Xin]{Yuan Lou and Ming-Zhen Xin}
\makeatletter
\g@addto@macro\@setauthors{%
  \par\vspace{-0.5\baselineskip}%
  \begin{center}
    \footnotesize
  School of Mathematical Sciences,\\
    Shanghai Jiao Tong University, Shanghai 200240, China.
  \end{center}%
}
\makeatother
\date{\today}

\begin{document}
\begin{abstract}
	
In this paper, we investigate how spatial heterogeneity affects the stability of positive equilibria in one-dimensional predator--prey reaction--diffusion systems with Neumann boundary conditions. Our examples show that spatial heterogeneity can lead to instability even when the positive equilibrium is unique. For the smooth, spatially heterogeneous coefficient families constructed here, this equilibrium loses linear stability as a parameter in the reaction coefficients varies. At the critical parameter value, a simple conjugate pair of eigenvalues crosses the imaginary axis transversely, and a Hopf bifurcation produces a local branch of strictly positive, nonstationary classical time-periodic solutions. The examples include two families with spatially concentrated prey, one with an everywhere positive and the other with an everywhere negative predator intrinsic growth rate. A third family has a spatially concentrated predator population and an everywhere negative predator intrinsic growth rate.
\vspace{1em}

\textbf{Keywords}: predator--prey system, positive equilibrium,
linear instability, spatial heterogeneity, Hopf bifurcation.

\textbf{AMS Subject Classification}: 35K57, 35B35, 35P15, 92D25.
\end{abstract}

\maketitle
\enlargethispage{4pt}

\section{Introduction and Main Results}

Classical predator--prey models date back to the work of
Lotka \cite{Lotka1920} and Volterra \cite{Volterra1926}, who used ordinary
differential equations to describe interacting populations.
In nature, individuals not only interact with one another but also
move through their environment. Skellam \cite{Skellam1951}
incorporated random dispersal into models of population growth
and considered spatially varying demographic coefficients.
Segel and Jackson \cite{SegelJackson1972} showed that diffusion can destabilize a spatially uniform equilibrium. This finding highlights the importance of studying diffusion together with local population interactions, for which reaction--diffusion systems provide a natural framework \cite{CantrellCosner19892003,OkuboLevin2001,LamLouBook}. For related work on predator--prey
systems, see
\cite{DuHsu2004,DuLou1997,DuLou1998,DuShi2006,KutoLopezGomezMunoz2026,LiWuDong2019,YiWeiShi2009,LopezNietoEtAl2025,PengWang2008,ShuXinWangYu2026}.

Spatial variation in food availability, shelter, and environmental
quality affects local growth rates as well as interaction rates.
These differences can be represented by position-dependent coefficients.
Earlier studies have shown that dispersal in heterogeneous environments
can affect population persistence, equilibrium abundance, and
competitive outcomes
\cite{CantrellCosner1989,CantrellCosner1991,Lou2006}.
In spatially heterogeneous predator--prey systems, diffusion also
influences species coexistence and the stability of positive
equilibria \cite{DuShi2007,DuShi2006}.

In this paper, we study the following predator–prey system with spatially varying reaction coefficients and Neumann boundary conditions:
\begin{equation}
\left\{
\begin{aligned}
&\partial_t u(t,x)
=d\,\partial_{xx}u(t,x)
+u(t,x)\bigl(r(x)-a(x)u(t,x)-b(x)v(t,x)\bigr),
&&t>0,\quad x\in(0,1),\\
&\partial_t v(t,x)
=D\,\partial_{xx}v(t,x)
+v(t,x)\bigl(m(x)+c(x)u(t,x)-e(x)v(t,x)\bigr),
&&t>0,\quad x\in(0,1),\\
&\partial_x u(t,x)=\partial_x v(t,x)=0,
&&t>0,\quad x\in\{0,1\}.
\end{aligned}
\right.
\label{eq:model}
\end{equation}
Here \(u,v\) denote the prey and predator densities, respectively;
\(d,D>0\) are their diffusion rates. The functions \(r,m\) specify intrinsic
growth rates; \(a,e\) are self-limitation coefficients, while \(b,c\)
represent predation loss and gain coefficients, respectively.
All six coefficients belong to \(C([0,1])\) and satisfy
\begin{equation}
a(x),b(x),c(x),e(x)>0,
\qquad x\in[0,1].
\label{eq:positive-coefficients}
\end{equation}
Under assumption~\eqref{eq:positive-coefficients}, L\'opez-G\'omez and Pardo
\cite[Theorem~4.1]{LopezGomezPardo1998} proved that system~\eqref{eq:model}
admits at most one positive equilibrium. This raises the question whether the positive equilibrium,
if it exists, must also be stable.

Let \(E=(u^*,v^*)\) be a positive equilibrium of \eqref{eq:model}.
Suppressing the spatial argument, the linearization at \(E\) is
\begin{equation}
	\mathcal L_E:=
	\begin{pmatrix}
		d\partial_{xx}+r-2au^*-bv^* & -bu^*\\[1mm]
		cv^* & D\partial_{xx}+m+cu^*-2ev^*
	\end{pmatrix}.
	\label{eq:linearized-operator}
\end{equation}
It acts on \(L^2((0,1))^2\), with domain
\[
\operatorname{Dom}(\mathcal L_E)
=\bigl\{(\xi,\psi)\in H^2((0,1))^2:
\xi'(0)=\xi'(1)=\psi'(0)=\psi'(1)=0\bigr\}.
\]
We omit \(\C\) from function-space notation in the spectral analysis,
where functions may be complex-valued. Real phase spaces are marked by \(\R\).
If the spectrum of \(\mathcal L_E\) lies strictly in the left
half-plane, then this unique equilibrium is linearly asymptotically
stable. For constant coefficients, \(\mathcal L_E\) is non-self-adjoint,
but expansion in Neumann eigenfunctions reduces its spectral analysis
to independent \(2\times2\) matrix problems. A direct calculation
then shows that every positive constant equilibrium is linearly
asymptotically stable for all \(d,D>0\).

When the coefficients vary in space, however, the equations obtained
from this expansion are generally coupled. The spectral problem
therefore no longer separates into independent matrix problems,
so the constant-coefficient calculation does not directly apply.
Determining stability in this setting requires a different approach
to the spectrum of \(\mathcal L_E\).
L\'opez-G\'omez and Pardo \cite[Section~4]{LopezGomezPardo1998}
showed that zero is not an eigenvalue of $\mathcal L_E$, identifying Hopf bifurcation as a possible mechanism for
loss of stability. Their work does not provide an example realizing
this mechanism. In this paper, we construct spatially heterogeneous systems in which
some reaction coefficients depend on a real parameter \(\eta\).
As \(\eta\) varies, the unique positive equilibrium loses stability
through a Hopf bifurcation, giving rise to a local branch of strictly
positive, nonstationary time-periodic solutions.
Thus uniqueness of a positive equilibrium neither guarantees its
stability nor rules out time-periodic dynamics.

To study the change of stability, we allow the reaction coefficients to
depend smoothly on a real parameter \(\eta\in[0,1]\), keeping the diffusion
rates fixed. In both results below, the coefficient families are smooth in
\((x,\eta)\in[0,1]^2\) and satisfy \eqref{eq:positive-coefficients}.

\begin{theorem}[Hopf bifurcation with spatially concentrated prey]
\label{thm:main}
For all sufficiently small \(\varepsilon>0\), there exist a coefficient
family independent of \(D\) and \(D_0(\varepsilon)>0\), with
\(D_0(\varepsilon)\to0\) as \(\varepsilon\to0\), such that, for
\(d=\varepsilon^2\), the following statements hold.
\begin{enumerate}
\item For every \(D>0\) and \(\eta\in[0,1]\), system~\eqref{eq:model}
has a unique positive equilibrium.
\item For every \(D>D_0(\varepsilon)\), this equilibrium loses linear
stability through a simple, transverse Hopf bifurcation as \(\eta\) increases through some
point in \((0,1)\), giving a local branch of strictly positive,
nonstationary classical time-periodic solutions.
\end{enumerate}
\end{theorem}

Now, we consider a numerical simulation to give the existence of a stable time-periodic solution. Take $d=1.30\times 10^{-2}$, $D=0.13$, and $\eta\in[0.90,0.95]$. The coefficient functions are given approximately by
\begingroup
\setlength{\abovedisplayskip}{6pt}
\setlength{\belowdisplayskip}{6pt}
\setlength{\abovedisplayshortskip}{6pt}
\setlength{\belowdisplayshortskip}{6pt}
\setlength{\jot}{6pt}
\[
\begin{aligned}
	&a(x)\approx 2.40\times10^{-2}e^{(1+\cos4\pi x)/2},
	\qquad e(x)\approx 9.62\times10^{-3},\\
	&b(x)\approx 10^{-2}e^{(1+\cos4\pi x)/2}
	\bigl(55.00e^{-20(x-1/4)^2}
	+6.96e^{-20(x-3/4)^2}\bigr),\\
	&c_\eta(x)\approx 10^{-2}
	\bigl[(9.23-2.27\eta)e^{-20(x-1/4)^2}
	+(55.38\eta-0.38)e^{-20(x-3/4)^2}\bigr],\\
	&r(x)\approx 3.61\times10^{-2}+1.50b(x)
	-\frac{13\pi^2}{250}
	\bigl(2\cos4\pi x+\sin^2 4\pi x\bigr),\\
	&m_\eta(x)\approx 1.44\times10^{-2}
	-1.50e^{-(1+\cos4\pi x)/2}c_\eta(x),\quad\quad x\in[0,1].
\end{aligned}
\]
\endgroup
The above values are rounded to two decimal places, and the calculations use the unrounded values. Under this set of coefficients, the unique positive equilibrium is
\[
u^*(x)=1.50e^{-(1+\cos4\pi x)/2},\qquad v^*(x)=1.50.
\]
We perform numerical computations using MATLAB and obtain the periodic solution illustrated in Figure ~\ref{fig:numerical-periodic-branch}. The Hopf bifurcation point is computed to be $\eta\approx 0.92$, with the corresponding period approximately $T=24.58$. Panels (a) and (b) show the evolution over four periods: the prey density remains spatially concentrated, whereas the predator density is nearly uniform in space. 
\begin{figure}[H]
	\centering
	\includegraphics[width=0.92\textwidth]{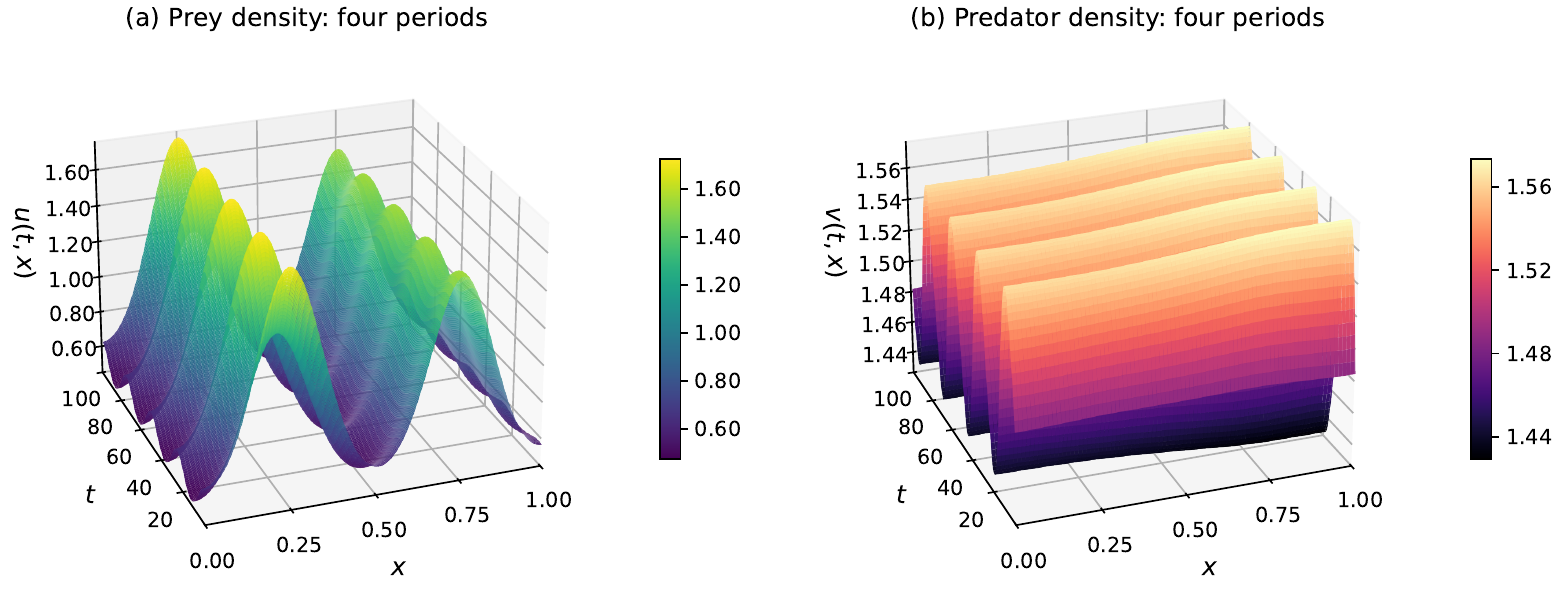}
	\caption{Numerical evolution of the periodic solution.}
	\label{fig:numerical-periodic-branch}
\end{figure}

\Needspace{5\baselineskip}
\begin{theorem}[Hopf bifurcation with spatially concentrated predator]
\label{thm:mirror-main}
For all sufficiently small \(\varepsilon>0\), there exist a coefficient
family independent of \(d\), with \(m_{\varepsilon,\eta}<0\), and a threshold
\(d_0(\varepsilon)>0\) satisfying \(d_0(\varepsilon)\to0\) as
\(\varepsilon\to0\), such that, for \(D=\varepsilon^2\), the following statements hold.
\begin{enumerate}
\item For every \(d>0\) and \(\eta\in[0,1]\), system~\eqref{eq:model}
has a unique positive equilibrium.
\item For every \(d>d_0(\varepsilon)\), this equilibrium loses linear
stability through a simple, transverse Hopf bifurcation as \(\eta\) increases through some
point in \((0,1)\), giving a local branch of strictly positive,
nonstationary classical time-periodic solutions.
\end{enumerate}
\end{theorem}

For a spatially concentrated predator population, we use a constant prey density and the same
eigenfunction as the predator density. The reconstructed linearization is obtained
from the corresponding operator in the first construction by taking its adjoint,
then exchanging its two components. The spectral conclusions therefore carry over.

The paper is organized as follows. Section~\ref{sec:preparations} establishes the scalar eigenvalue estimates and
positive coupling functions used in both constructions.
Section~\ref{sec:concentrated-prey} treats spatially concentrated prey
and proves Theorem~\ref{thm:main}.
Section~\ref{sec:concentrated-predator} treats a spatially concentrated predator population
and proves Theorem~\ref{thm:mirror-main}.

\section{Spectral Preparation and Coupling Functions}
\label{sec:preparations}

In this section, we establish the scalar spectral estimates and construct the positive coupling
functions used in both instability constructions.

\subsection{Localized eigenfunctions and spectral estimates}
\label{subsec:scalar-spectrum}

We begin by establishing the scalar eigenvalue gaps and eigenfunction localization
estimates needed for the later reduction.

Fix a symmetric function
\(W\in C^\infty([0,1])\) satisfying
\begin{equation}
W(x)\ge0,\quad W^{-1}(0)=\{x_1,x_2\},\quad
0<x_1<x_2<1,\quad W(1-x)=W(x),
\quad\text{and}\quad W''(x_j)>0.
\label{eq:double-well}
\end{equation}
By symmetry, we have \(x_2=1-x_1\). Fix \(\ell>0\) such that
\(3\ell<\min\{x_1,1/2-x_1\}\).
Set \(J:=(x_1-3\ell,x_1+3\ell)\).
On \(J\), define the Dirichlet operator
\begin{equation}
\begin{aligned}
\cH_{\varepsilon,0}^{\mathrm D}w&=-\varepsilon^2w''+Ww,\quad\text{with}
\operatorname{Dom}(\cH_{\varepsilon,0}^{\mathrm D})
=H^2(J)
\cap H_0^1(J).
\end{aligned}
 \label{eq:local-Dirichlet-operator}
\end{equation}
Here
\[
H_0^1(J)
:=\{w\in H^1(J):w(x_1-3\ell)=w(x_1+3\ell)=0\}.
\]
Denote its \(k\)-th eigenvalue by \(\mu_{k,\varepsilon}\).

The limiting one-dimensional operator is
\begin{equation}
 \cH_0
 :=-\frac{\dd^2}{\dd y^2}+\frac{W''(x_1)}2y^2
 \quad\text{on }L^2(\mathbb R).
 \label{eq:limiting-quadratic-operator}
\end{equation}
Let \(\nu_1<\nu_2\) be its first two eigenvalues.

\begin{lemma}[Local eigenvalue limit near \(x_1\)]
\label{lem:half-local-spectrum}
For \(k=1,2\),
\[
\lim_{\varepsilon\to0^+}
\frac{\mu_{k,\varepsilon}}{\varepsilon}=\nu_k.
\]
\end{lemma}

\begin{proof}
Extend \(W\) from \(\overline J\) to a nonnegative
\(\widetilde W\in C^\infty(\mathbb R)\) with \(x_1\) as its only zero
and \(\widetilde W=1\) outside a compact set. Such an extension exists
because \(W>0\) away from \(x_1\) on this interval.
Moreover, \(\widetilde W''(x_1)=W''(x_1)>0\).

Choose \(C_*\) strictly between the second and third eigenvalues of
\(\cH_0\). The extension satisfies hypotheses (A1)--(A4) of
\cite{Simon1983}. Applying \cite[Theorem~1.1]{Simon1983} with
\(h=\widetilde W\), \(g=0\), \(\lambda=\varepsilon^{-1}\), we find that
\(-\varepsilon^2\partial_{xx}+\widetilde W\) on \(L^2(\mathbb R)\)
has exactly two eigenvalues in \([0,C_*\varepsilon]\) for small
\(\varepsilon\). Denoting them in increasing order by
\(\widetilde\mu_{1,\varepsilon}\), \(\widetilde\mu_{2,\varepsilon}\), we have
\begin{equation}
\lim_{\varepsilon\to0^+}
\frac{\widetilde\mu_{k,\varepsilon}}{\varepsilon}=\nu_k,
\qquad k=1,2.
\label{eq:extended-local-eigenvalue-limit}
\end{equation}
\Needspace{5\baselineskip}
Applying \cite[Lemma~3.4.3]{Helffer1988} to
\(J\subset\mathbb R\), with \(h=\varepsilon\),
gives exactly two local Dirichlet eigenvalues in
\([0,C_*\varepsilon]\), satisfying, for some \(C,c>0\) independent of
\(\varepsilon\),
\begin{equation}
|\mu_{k,\varepsilon}-\widetilde\mu_{k,\varepsilon}|
\le Ce^{-c/\varepsilon},\qquad k=1,2.
\label{eq:local-domain-comparison}
\end{equation}
The conclusion follows from \eqref{eq:extended-local-eigenvalue-limit}
and \eqref{eq:local-domain-comparison}.
\end{proof}

The eigenfunctions and test functions below may be chosen real-valued.
Define the half-interval operators by
\[
 \bigl(\cH_\varepsilon^{\pm}w\bigr)(x)
 :=-\varepsilon^2w''(x)+W(x)w(x),
 \quad x\in(0,1/2),
\]
with the boundary conditions
\[
 w'(0)=w'(1/2)=0\quad\text{for }\cH_\varepsilon^+,
 \qquad
 w'(0)=w(1/2)=0\quad\text{for }\cH_\varepsilon^-.
\]
The operators \(\cH_\varepsilon^\pm\) are self-adjoint on
\(L^2((0,1/2))\), with domains consisting of the
\(H^2((0,1/2))\) functions satisfying the respective boundary conditions.
Their eigenvalues are real; since the coefficients and boundary
conditions are real, their eigenfunctions can be chosen real-valued.
Define
\[
 d_1(x):=\left|\int_{x_1}^{x}\sqrt{W(s)}\,\dd s\right|,
 \qquad x\in[0,1/2].
\]
\Needspace{5\baselineskip}
\begin{lemma}[Weighted decay]
\label{lem:half-Agmon}
For \(\mathcal H\in\{\cH_\varepsilon^+,\cH_\varepsilon^-\}\)
and fixed \(C_0>0\), suppose that
\begin{equation}
 \mathcal H\psi_\varepsilon=\theta_\varepsilon\psi_\varepsilon,
 \qquad \|\psi_\varepsilon\|_{L^2((0,1/2))}=1,
 \quad\text{and}\quad 0\le\theta_\varepsilon\le C_0\varepsilon.
\label{eq:weighted-eigenvalue-equation}
\end{equation}
If \(K\subset[0,1/2]\setminus\{x_1\}\) is a nonempty compact set and
\(0<\kappa<\min_{x\in K}d_1(x)\), then there exists
\(C>0\), depending on \(K,\kappa,C_0\), such that, for
all sufficiently small \(\varepsilon\),
\begin{equation}
 \|\psi_\varepsilon\|_{L^2(K)}
 +\varepsilon\|\psi_\varepsilon'\|_{L^2(K)}
 \le C\exp(-\kappa/\varepsilon).
\label{eq:direct-half-Agmon}
\end{equation}
\end{lemma}

\begin{proof}
At each endpoint of \((0,1/2)\), either
\(\psi_\varepsilon=0\) or \(\psi_\varepsilon'=0\). Thus the boundary term
in the weighted integration by parts used in
\cite[Theorem~3.1.1]{Helffer1988} vanishes:
\[
-\varepsilon^2
\left[e^{2\Phi/\varepsilon}\psi_\varepsilon'
\overline{\psi_\varepsilon}\right]_0^{1/2}=0
\]
for every real Lipschitz weight \(\Phi\).
Consequently, the proof of \cite[Proposition~3.3.1]{Helffer1988}
applies to both half-interval operators with \(h=\varepsilon\),
\(E=0\), and \(U=\{x_1\}\).
Put \(d_K=\min_{x\in K}d_1(x)\) and choose \(0<\rho<d_K-\kappa\).
Since \(d_1'\) is bounded, the product rule applied to the weighted
estimate (3.3.2) in that reference gives, for some
\(C_\rho>0\) independent of \(\varepsilon\),
\[
\|\psi_\varepsilon\|_{L^2(K)}
+\varepsilon\|\psi_\varepsilon'\|_{L^2(K)}
\le C_\rho e^{-(d_K-\rho)/\varepsilon}
\le C_\rho e^{-\kappa/\varepsilon}.
\qedhere
\]
\end{proof}

For each sign, let \(\{\alpha_{k,\varepsilon}^{\pm}\}_{k\ge1}\) be the
eigenvalues of \(\cH_\varepsilon^\pm\), listed in increasing order.

\begin{lemma}[Comparison of the first two eigenvalues]
\label{lem:Dirichlet-half-comparison}
There exist constants \(C_1,\kappa_1>0\)
such that, for all sufficiently small \(\varepsilon\),
\begin{equation}
\left|\alpha_{k,\varepsilon}^{\pm}-\mu_{k,\varepsilon}\right|
\le C_1e^{-\kappa_1/\varepsilon},
\qquad k=1,2.
\label{eq:direct-half-comparison}
\end{equation}
\end{lemma}

\begin{proof}
Choose \(\chi\in C_c^\infty(J)\) such that
\(0\le\chi\le1\) and
\begin{equation}
\chi(x)=1\quad\text{for }|x-x_1|\le\ell,
\qquad
\chi(x)=0\quad\text{for }|x-x_1|\ge2\ell.
\label{eq:half-cutoff-function}
\end{equation}
Extend \(\chi\) by zero to \((0,1/2)\). Set
\begin{equation}
 K_\chi:=\operatorname{supp}_{[0,1/2]}(1-\chi^2)
 \cup\operatorname{supp}\chi'
 \Subset[0,1/2]\setminus\{x_1\}.
\label{eq:half-cutoff-support}
\end{equation}
Choose \(0<\kappa_1<\min_{x\in K_\chi}d_1(x)\).

For \(w\in H_0^1(J)\), define its extension to
\((0,1/2)\) by
\begin{equation}
\widetilde w(x):=
\begin{cases}
w(x),&|x-x_1|<3\ell,\\
0,&|x-x_1|\ge3\ell.
\end{cases}
\label{eq:zero-extension}
\end{equation}
Then
\[
\widetilde w
\in\{v\in H^1((0,1/2)):v(1/2)=0\}\subset H^1((0,1/2)).
\]
Moreover,
\begin{subequations}
\label{eq:zero-extension-identities}
\begin{equation}
\|\widetilde w\|_{L^2((0,1/2))}^2
=\|w\|_{L^2(J)}^2,
\label{eq:zero-extension-L2}
\end{equation}
and
\begin{equation}
\int_0^{1/2}\bigl(\varepsilon^2|\widetilde w'|^2
+W|\widetilde w|^2\bigr)\,\dd x
=\int_J
\bigl(\varepsilon^2|w'|^2+W|w|^2\bigr)\,\dd x.
\label{eq:zero-extension-integral}
\end{equation}
\end{subequations}
By \eqref{eq:zero-extension-L2}--\eqref{eq:zero-extension-integral},
the extension preserves the Rayleigh quotient and the dimension of every
trial subspace. The min--max principle
\cite[Theorem~A.1.6]{FournaisHelffer2010}, together with
Lemma~\ref{lem:half-local-spectrum}, therefore gives, for small
\(\varepsilon\),
\begin{equation}
\alpha_{j,\varepsilon}^{\pm}\le\mu_{j,\varepsilon}
\le2\nu_2\varepsilon,\qquad j=1,2.
\label{eq:half-eigenvalue-upper-bound}
\end{equation}
Let \(F_{\pm,k}\) be the span of the first \(k\) \(L^2((0,1/2))\)-normalized
eigenfunctions of \(\cH_\varepsilon^\pm\). Applying
Lemma~\ref{lem:half-Agmon} to these eigenfunctions and then taking a
linear combination gives, with \(C_1\) chosen to include the fixed cutoff
factors,
\begin{equation}
\|w\|_{L^2(K_\chi)}+\varepsilon\|w'\|_{L^2(K_\chi)}
\le C_1e^{-\kappa_1/\varepsilon}\|w\|_{L^2((0,1/2))},
\qquad w\in F_{\pm,k}.
\label{eq:half-trial-space-decay}
\end{equation}
For \(w\in F_{\pm,k}\), \eqref{eq:half-trial-space-decay} and
\(\operatorname{supp}(1-\chi^2)\subset K_\chi\) give
\begin{subequations}
\begin{equation}
\begin{aligned}
\|w\|_{L^2((0,1/2))}^2
-\|\chi w\|_{L^2(J)}^2
&=\int_0^{1/2}(1-\chi^2)|w|^2\,\dd x\le C_1e^{-\kappa_1/\varepsilon}\|w\|_{L^2((0,1/2))}^2.
\end{aligned}
\label{eq:cutoff-norm}
\end{equation}
Next, expand \((\chi w)'=\chi w'+\chi'w\). In the difference of the two
integrals, all terms are supported in \(K_\chi\). The bounds in
\eqref{eq:half-trial-space-decay}, combined with the Cauchy--Schwarz
inequality, therefore give
\begin{equation}
\begin{aligned}
&\left|\int_J
(\varepsilon^2|(\chi w)'|^2+W|\chi w|^2)\,\dd x
-\int_0^{1/2}(\varepsilon^2|w'|^2+W|w|^2)\,\dd x\right|\le C_1e^{-\kappa_1/\varepsilon}\|w\|_{L^2((0,1/2))}^2.
\end{aligned}
\label{eq:cutoff-energy}
\end{equation}
\end{subequations}
Estimate \eqref{eq:cutoff-norm} shows that
\(w\mapsto\chi w\) is injective on \(F_{\pm,k}\) for small
\(\varepsilon\), and
\(\chi F_{\pm,k}\subset H_0^1(J)\).
Using \(\chi F_{\pm,k}\) in the \(k\)-th-eigenvalue formula for
\(\cH_{\varepsilon,0}^{\mathrm D}\), together with
\eqref{eq:cutoff-norm} and \eqref{eq:cutoff-energy}, gives, for small
\(\varepsilon\),
\[
\begin{aligned}
\mu_{k,\varepsilon}
&\le\sup_{w\in F_{\pm,k}\setminus\{0\}}
\frac{\int_J\bigl(\varepsilon^2|(\chi w)'|^2+W|\chi w|^2\bigr)\,\dd x}
{\|\chi w\|_{L^2(J)}^2}\\
&\le\frac{\alpha_{k,\varepsilon}^{\pm}+C_1e^{-\kappa_1/\varepsilon}}
{1-C_1e^{-\kappa_1/\varepsilon}}
\le\alpha_{k,\varepsilon}^{\pm}+4C_1e^{-\kappa_1/\varepsilon}.
\end{aligned}
\]
After increasing \(C_1\), combining this with
\(\alpha_{k,\varepsilon}^{\pm}\le\mu_{k,\varepsilon}\) proves
\eqref{eq:direct-half-comparison}.
\end{proof}

We now turn to the Neumann operator on the full interval,
\begin{equation}
\cH_\varepsilon w=-\varepsilon^2w''+Ww,
\qquad
\operatorname{Dom}(\cH_\varepsilon)=H_N^2((0,1)),
\label{eq:schrodinger}
\end{equation}
where
\[
H_N^2((0,1))
:=\{w\in H^2((0,1)):w'(0)=w'(1)=0\}.
\]
Let \(\lambda_{1,\varepsilon}<\lambda_{2,\varepsilon}<\lambda_{3,\varepsilon}\)
be its first three eigenvalues. Set
\begin{equation}
\Lambda_{1,\varepsilon}:=\lambda_{2,\varepsilon}-\lambda_{1,\varepsilon}
\quad\text{and}\quad
\Lambda_{2,\varepsilon}:=\lambda_{3,\varepsilon}-\lambda_{1,\varepsilon}.
\label{eq:spectral-gaps}
\end{equation}

\begin{lemma}[First three eigenvalues]
\label{lem:symmetric-double-well}
For all sufficiently small \(\varepsilon\),
\begin{equation}
0<\lambda_{1,\varepsilon}\le2\nu_2\varepsilon,
\qquad
0<\Lambda_{1,\varepsilon}
\le2C_1e^{-\kappa_1/\varepsilon}
\quad\text{and}\quad
\Lambda_{2,\varepsilon}\ge\frac{\nu_2-\nu_1}{4}\varepsilon.
\label{eq:scale-separation}
\end{equation}
\end{lemma}

\begin{proof}
Since \(W(x)=W(1-x)\), reflection about \(x=1/2\) gives the following
decomposition into the two half-interval problems.
Define the isometries
\(\mathcal U_\pm:L^2((0,1/2))\to L^2((0,1))\) by
\[
 (\mathcal U_\pm w)(x):=\frac1{\sqrt2}
 \begin{cases}
  w(x),&x\in(0,1/2),\\
  \pm w(1-x),&x\in(1/2,1).
\end{cases}
\]
Let
\(\mathcal U:L^2((0,1/2))\oplus L^2((0,1/2))
\to L^2((0,1))\) be
\(\mathcal U(w_+,w_-)=\mathcal U_+w_++\mathcal U_-w_-\).
The map \(\mathcal U\) preserves the \(L^2\) inner product and maps onto
\(L^2((0,1))\). Moreover,
\begin{equation}
 \mathcal U^*\cH_\varepsilon\mathcal U
 =\cH_\varepsilon^+\oplus\cH_\varepsilon^-.
 \label{eq:half-interval-decomposition}
\end{equation}
Replacing the Dirichlet condition at \(x=1/2\) by a Neumann
condition strictly decreases the first eigenvalue by
\cite[Theorem~17, p.~18]{Amann2005}; hence
\(\alpha_{1,\varepsilon}^{+}<\alpha_{1,\varepsilon}^{-}\).
With \(C_1,\kappa_1\) from Lemma~\ref{lem:Dirichlet-half-comparison},
\eqref{eq:direct-half-comparison} gives
\begin{equation}
0<\alpha_{1,\varepsilon}^{-}-\alpha_{1,\varepsilon}^{+}
\le2C_1e^{-\kappa_1/\varepsilon}=o(\varepsilon).
\label{eq:half-ground-splitting}
\end{equation}
Lemmas~\ref{lem:half-local-spectrum} and
\ref{lem:Dirichlet-half-comparison} also give
\begin{equation}
\begin{aligned}
\alpha_{2,\varepsilon}^{\pm}-\alpha_{1,\varepsilon}^{\pm}
&\ge\mu_{2,\varepsilon}-\mu_{1,\varepsilon}
-2C_1e^{-\kappa_1/\varepsilon}\\
&=\varepsilon\bigl(\nu_2-\nu_1+o(1)\bigr)
\ge\frac{\nu_2-\nu_1}{4}\varepsilon.
\end{aligned}
\label{eq:half-eigenvalue-gap}
\end{equation}
Equations \eqref{eq:half-ground-splitting} and
\eqref{eq:half-eigenvalue-gap} give, for all
sufficiently small \(\varepsilon\),
\[
\alpha_{1,\varepsilon}^{-}
<\min\{\alpha_{2,\varepsilon}^{+},\alpha_{2,\varepsilon}^{-}\}.
\]
By \eqref{eq:half-interval-decomposition}, the first three eigenvalues of
\(\cH_\varepsilon\) satisfy
\begin{equation}
\lambda_{1,\varepsilon}=\alpha_{1,\varepsilon}^{+},
\qquad
\lambda_{2,\varepsilon}=\alpha_{1,\varepsilon}^{-},
\qquad\text{and}\qquad
 \lambda_{3,\varepsilon}
 =\min\{\alpha_{2,\varepsilon}^{+},\alpha_{2,\varepsilon}^{-}\}.
\label{eq:full-half-spectrum-ordering}
\end{equation}
Since \(W\ge0\) and \(W\not\equiv0\), we have
\(\lambda_{1,\varepsilon}>0\).
The bound for \(\Lambda_{1,\varepsilon}\) follows from
\eqref{eq:half-ground-splitting} and \eqref{eq:full-half-spectrum-ordering}.
The upper bound for
\(\lambda_{1,\varepsilon}\) follows from
\(\alpha_{1,\varepsilon}^{+}\le\mu_{1,\varepsilon}\) and
the second inequality in \eqref{eq:half-eigenvalue-upper-bound}. Finally,
\eqref{eq:half-eigenvalue-gap} and
\eqref{eq:full-half-spectrum-ordering} give
\[
\Lambda_{2,\varepsilon}
\ge\min\{
\alpha_{2,\varepsilon}^{+}-\alpha_{1,\varepsilon}^{+},
\alpha_{2,\varepsilon}^{-}-\alpha_{1,\varepsilon}^{-}\}
\ge\frac{\nu_2-\nu_1}{4}\varepsilon.
\]
This proves \eqref{eq:scale-separation}.
\end{proof}

\subsection{A target matrix and positive coupling functions}
\label{subsec:target-couplings}

We next choose the target matrix and construct strictly positive coupling functions
with prescribed inner products.

Choose real \(L^2((0,1))\)-normalized eigenfunctions
\(\phi_{1,\varepsilon},\phi_{2,\varepsilon}\) of \(\cH_\varepsilon\)
corresponding to \(\lambda_{1,\varepsilon},\lambda_{2,\varepsilon}\),
with \(\phi_{1,\varepsilon}>0\) and \(\phi_{2,\varepsilon}(x_1)>0\).
Define the two combinations
\begin{equation}
\zeta_{1,\varepsilon}
:=\frac{\phi_{1,\varepsilon}+\phi_{2,\varepsilon}}{\sqrt2}
\quad\text{and}\quad
\zeta_{2,\varepsilon}
:=\frac{\phi_{1,\varepsilon}-\phi_{2,\varepsilon}}{\sqrt2}.
\label{eq:localized-modes}
\end{equation}

For the concentrated-prey construction, we use the combinations
\(\zeta_{1,\varepsilon},\zeta_{2,\varepsilon}\) from
\eqref{eq:localized-modes} as the two
selected prey functions and \(\1\) as the predator function. Since
\(\partial_{xx}\1=0\), this third function has zero diffusion term, whereas
every nonconstant Neumann predator eigenfunction has diffusion eigenvalue at most
\(-D\pi^2\). We therefore select
\begin{equation}
\cX_\varepsilon
=\operatorname{span}_{\C}
\bigl\{(\zeta_{1,\varepsilon},0),
(\zeta_{2,\varepsilon},0),(0,\1)\bigr\}.
\label{eq:slow-space}
\end{equation}
We next design a finite-dimensional crossing on the three functions in
\eqref{eq:slow-space}.
Fix \(L>4\), and choose
\begin{equation}
0<\gamma_0<\min\left\{1,\frac L2-1\right\}.
\label{eq:Lgamma-range}
\end{equation}
We prescribe the target matrix
\begin{equation}
M_\gamma(\eta):=
\begin{pmatrix}
-1&1/2&-L\\
1/2&-1&-1\\
1&\eta L&-\gamma
\end{pmatrix},
\qquad 0\le\gamma\le\gamma_0,\quad \eta>0.
\label{eq:target-three-mode-matrix}
\end{equation}
The coefficient choices in \eqref{eq:block-requirements} will realize its
diagonal blocks, while the four coupling moments below will realize the
remaining row and column; the exact compression is verified in
\eqref{eq:compression}.
For \(0\le\gamma\le\gamma_0\), set
\begin{equation}
\eta_*(\gamma)
:=\frac{2\bigl(1+4\gamma+2\gamma^2+L(1+\gamma)\bigr)}
{L\bigl(L-2(1+\gamma)\bigr)},
\qquad
\omega_\gamma
:=\sqrt{\tfrac34+2\gamma+L(1+\eta_*(\gamma))}.
\label{eq:matrix-threshold}
\end{equation}
Writing \(I_3\) for the \(3\times3\) identity matrix, a direct calculation gives
\begin{equation}
\begin{aligned}
\det(zI_3-M_\gamma(\eta))
={}&(z+2+\gamma)(z^2+\omega_\gamma^2)\\
&+L(\eta-\eta_*(\gamma))(z+1+L/2).
\end{aligned}
\label{eq:target-characteristic-polynomial}
\end{equation}
Writing \(\sigma(T)\) for the spectrum of an operator or matrix \(T\),
we obtain
\begin{equation}
\sigma\bigl(M_\gamma(\eta_*(\gamma))\bigr)
=\{-(2+\gamma),\,i\omega_\gamma,\,-i\omega_\gamma\}.
\label{eq:target-critical-spectrum}
\end{equation}
These eigenvalues are simple. Differentiating
\eqref{eq:target-characteristic-polynomial} at the root
\(z_{+,\gamma}(\eta_*(\gamma))=i\omega_\gamma\) yields
\begin{equation}
\left.\frac{\dd}{\dd\eta}\operatorname{Re} z_{+,\gamma}(\eta)
\right|_{\eta=\eta_*(\gamma)}
=\frac{L\bigl(L-2(1+\gamma)\bigr)}
{4\bigl(\omega_\gamma^2+(2+\gamma)^2\bigr)}>0.
\label{eq:target-crossing-direction}
\end{equation}
Thus the conjugate pair crosses from left to right, while the third
eigenvalue remains negative near \(\eta_*(\gamma)\).

The explicit formula in \eqref{eq:matrix-threshold} shows that \(\eta_*\) is
positive and continuous on \([0,\gamma_0]\). We may therefore choose an interval
\begin{equation}
I=[\eta_{\mathrm L},\eta_{\mathrm R}]\Subset(0,\infty),\qquad
\eta_*\bigl([0,\gamma_0]\bigr)\Subset(\eta_{\mathrm L},\eta_{\mathrm R}).
\label{eq:coupling-parameter-interval}
\end{equation}
Fix \(I\) for the remainder of the paper.

\Needspace{12\baselineskip}
\begin{lemma}[Positive coupling functions]
\label{lem:positive-profiles}
For all sufficiently small \(\varepsilon\), there are strictly positive functions
\(k_{b,\varepsilon},k_{c,\varepsilon,\eta}\in C^\infty([0,1])\), depending smoothly on
\(\eta\in I\),
such that
\begin{equation}
\begin{pmatrix}
\langle k_{b,\varepsilon},\zeta_{1,\varepsilon}\rangle&
\langle k_{b,\varepsilon},\zeta_{2,\varepsilon}\rangle\\
\langle k_{c,\varepsilon,\eta},\zeta_{1,\varepsilon}\rangle&
\langle k_{c,\varepsilon,\eta},\zeta_{2,\varepsilon}\rangle
\end{pmatrix}
=
\begin{pmatrix}L&1\\1&\eta L\end{pmatrix},
\label{eq:moments}
\end{equation}
and
\begin{equation}
K_\varepsilon:=
1+\sup_{\eta\in I}
\bigl(
\|k_{b,\varepsilon}\|_{C^1([0,1])}
+\|k_{c,\varepsilon,\eta}\|_{C^1([0,1])}
+\|\partial_\eta k_{c,\varepsilon,\eta}\|_{C^1([0,1])}
\bigr)
 \le C_{\mathrm{pr}}\varepsilon^{-2},
\label{eq:profile-bound}
\end{equation}
where \(C_{\mathrm{pr}}>0\) is independent of \(\varepsilon,\eta\).
\end{lemma}

\begin{proof}
Let \(g_\varepsilon^\pm\) be the positive
\(L^2((0,1/2))\)-normalized first eigenfunctions of \(\cH_\varepsilon^\pm\).
Equation \eqref{eq:half-interval-decomposition}, together with the first two identities in
\eqref{eq:full-half-spectrum-ordering}, gives
\(\phi_{1,\varepsilon}=g_\varepsilon^+/\sqrt2\) and
\(\phi_{2,\varepsilon}=g_\varepsilon^-/\sqrt2\) on \((0,1/2)\).
The signs follow from \(g_\varepsilon^\pm>0\),
\(\phi_{1,\varepsilon}>0\), and \(\phi_{2,\varepsilon}(x_1)>0\).
Thus \eqref{eq:localized-modes} gives, on this interval,
\[
\zeta_{1,\varepsilon}(x)=\tfrac12(g_\varepsilon^+(x)+g_\varepsilon^-(x)),\qquad
\zeta_{2,\varepsilon}(x)=\tfrac12(g_\varepsilon^+(x)-g_\varepsilon^-(x)).
\]
Reflection also gives
\begin{equation}
\zeta_{1,\varepsilon}(1-x)=\zeta_{2,\varepsilon}(x),\qquad x\in(0,1).
\label{eq:zeta-reflection}
\end{equation}
The positivity of \(g_\varepsilon^\pm\) gives
\(\zeta_{1,\varepsilon}>0\) on \((0,1/2)\); reflection gives
\(\zeta_{2,\varepsilon}>0\) on \((1/2,1)\).
By \eqref{eq:half-eigenvalue-upper-bound},
\(\alpha_{1,\varepsilon}^{\pm}\le2\nu_2\varepsilon\).
Lemma~\ref{lem:half-Agmon}, applied to \(g_\varepsilon^\pm\) on
\(K:=[0,x_1-\ell]\cup[x_1+\ell,1/2]\), gives constants
\(C,\kappa>0\), independent of \(\varepsilon\), such that
\(\int_K|g_\varepsilon^\pm|^2\,\dd x\le C^2e^{-2\kappa/\varepsilon}\).
Normalization therefore gives, for small \(\varepsilon\),
\begin{equation}
\int_{x_1-\ell}^{x_1+\ell}|g_\varepsilon^\pm|^2\,\dd x
=1-\int_K|g_\varepsilon^\pm|^2\,\dd x
\ge1-C^2e^{-2\kappa/\varepsilon}\ge\frac12.
\label{eq:direct-half-first-eigenfunction-lower-bound}
\end{equation}
Consequently, positivity, \eqref{eq:zeta-reflection}, and
\eqref{eq:direct-half-first-eigenfunction-lower-bound} imply that, for all
sufficiently small \(\varepsilon\),
\begin{equation}
\int_{x_j-\ell}^{x_j+\ell}|\zeta_{j,\varepsilon}|^2\,\dd x
\ge\frac18,
\qquad j=1,2.
\label{eq:own-well-mass}
\end{equation}
Choose \(\chi\in C_c^\infty((x_1-2\ell,x_1+2\ell))\) so that
\(0\le\chi\le1\) and \(\chi=1\) on \((x_1-\ell,x_1+\ell)\). Set
\(\chi^\sharp(x):=\chi(1-x)\).  Reflection symmetry and
\eqref{eq:own-well-mass} give
\[
d_\varepsilon
:=\langle\chi\zeta_{1,\varepsilon},\zeta_{1,\varepsilon}\rangle
 =\langle\chi^\sharp\zeta_{2,\varepsilon},
           \zeta_{2,\varepsilon}\rangle
 \ge\frac18.
\]
Define
\[
\rho_{1,\varepsilon}
:=\frac{\chi\zeta_{1,\varepsilon}}{d_\varepsilon}
\quad\text{and}\quad
\rho_{2,\varepsilon}
:=\frac{\chi^\sharp\zeta_{2,\varepsilon}}{d_\varepsilon}.
\]
Both functions are nonnegative. Moreover,
\begin{subequations}
\label{eq:rho-moments}
\begin{equation}
\langle\rho_{1,\varepsilon},\zeta_{1,\varepsilon}\rangle
=\langle\rho_{2,\varepsilon},\zeta_{2,\varepsilon}\rangle=1.
\label{eq:rho-diagonal-moments}
\end{equation}
Their two cross moments agree:
\begin{equation}
q_\varepsilon
:=\langle\rho_{1,\varepsilon},\zeta_{2,\varepsilon}\rangle
=\langle\rho_{2,\varepsilon},\zeta_{1,\varepsilon}\rangle.
\label{eq:rho-cross-moments}
\end{equation}
\end{subequations}
Since \(\int_0^{1/2}(g_\varepsilon^+)^2\,\dd x
=\int_0^{1/2}(g_\varepsilon^-)^2\,\dd x=1\),
\eqref{eq:rho-cross-moments} gives
\[
\begin{aligned}
q_\varepsilon
&=\frac1{4d_\varepsilon}\int_0^{1/2}
\chi\bigl[(g_\varepsilon^+)^2-(g_\varepsilon^-)^2\bigr]\,\dd x\\
&=-\frac1{4d_\varepsilon}\int_0^{1/2}
(1-\chi)\bigl[(g_\varepsilon^+)^2-(g_\varepsilon^-)^2\bigr]\,\dd x.
\end{aligned}
\]
Using \(\operatorname{supp}_{[0,1/2]}(1-\chi)\subset K\),
\(d_\varepsilon\ge1/8\), and Lemma~\ref{lem:half-Agmon}, we obtain
\[
|q_\varepsilon|
\le\frac1{4d_\varepsilon}\int_K
\bigl[(g_\varepsilon^+)^2+(g_\varepsilon^-)^2\bigr]\,\dd x
\le4C^2e^{-2\kappa/\varepsilon}\longrightarrow0.
\]

For small \(\varepsilon\), the matrix of the inner products in
\eqref{eq:rho-moments} and
its inverse are
\[
G_\varepsilon
:=\begin{pmatrix}1&q_\varepsilon\\q_\varepsilon&1\end{pmatrix},
\qquad
G_\varepsilon^{-1}
=\frac1{1-q_\varepsilon^2}
\begin{pmatrix}1&-q_\varepsilon\\-q_\varepsilon&1\end{pmatrix}
\longrightarrow\mathrm{Id}.
\]
Choose \(0<\delta_0<\frac12\min\{1,\eta_{\mathrm L}L\}\).
Define the two target vectors by
\[
\mathbf y_{b,\varepsilon}
:=\binom{L}{1}
-\delta_0\binom{\langle\1,\zeta_{1,\varepsilon}\rangle}
 {\langle\1,\zeta_{2,\varepsilon}\rangle}
\quad\text{and}\quad
\mathbf y_{c,\varepsilon}(\eta)
:=\binom{1}{\eta L}
-\delta_0\binom{\langle\1,\zeta_{1,\varepsilon}\rangle}
 {\langle\1,\zeta_{2,\varepsilon}\rangle}.
\]
Because
\(|\langle\1,\zeta_{j,\varepsilon}\rangle|\le1\), the target vectors are
uniformly bounded, with every component at least
\(\tfrac12\min\{1,\eta_{\mathrm L}L\}>0\), for \(\eta\in I\).
Since \(G_\varepsilon^{-1}\to\mathrm{Id}\), the vectors
\[
\boldsymbol\beta_{b,\varepsilon}
:=G_\varepsilon^{-1}\mathbf y_{b,\varepsilon},
\qquad\text{and}\qquad
\boldsymbol\beta_{c,\varepsilon}(\eta)
:=G_\varepsilon^{-1}\mathbf y_{c,\varepsilon}(\eta)
\]
have positive components for every \(\eta\in I\) when \(\varepsilon\)
is small. Write their components as
\(\beta_{b,\varepsilon,j}\) and
\(\beta_{c,\varepsilon,j}(\eta)\), \(j=1,2\). Set
\begin{align}
k_{b,\varepsilon}
&:=\delta_0\1+\beta_{b,\varepsilon,1}\rho_{1,\varepsilon}
+\beta_{b,\varepsilon,2}\rho_{2,\varepsilon},
\label{eq:kb-direct-definition}\\
k_{c,\varepsilon,\eta}
&:=\delta_0\1+\beta_{c,\varepsilon,1}(\eta)\rho_{1,\varepsilon}
+\beta_{c,\varepsilon,2}(\eta)\rho_{2,\varepsilon}.
\label{eq:kc-direct-definition}
\end{align}
Both functions are at least \(\delta_0\).
By \eqref{eq:rho-moments}, the identities
\(G_\varepsilon\boldsymbol\beta_{b,\varepsilon}=\mathbf y_{b,\varepsilon}\)
and
\(G_\varepsilon\boldsymbol\beta_{c,\varepsilon}(\eta)
=\mathbf y_{c,\varepsilon}(\eta)\)
give \eqref{eq:moments}.
The second target vector is affine in \(\eta\);
hence so is \(k_{c,\varepsilon,\eta}\).
Finally, by \eqref{eq:schrodinger},
\eqref{eq:half-eigenvalue-upper-bound}, and the first two identities
in \eqref{eq:full-half-spectrum-ordering}, the normalized eigenfunctions
satisfy, for \(j=1,2\),
\[
\|\phi_{j,\varepsilon}''\|_{L^2((0,1))}
\le\varepsilon^{-2}(\|W\|_{L^\infty((0,1))}+\lambda_{j,\varepsilon})
\le(\|W\|_{L^\infty((0,1))}+2\nu_2)\varepsilon^{-2},
\qquad 0<\varepsilon\le1.
\]
The standard one-dimensional \(H^2\) estimate, the embedding
\(H^2((0,1))\hookrightarrow C^1([0,1])\), the fixed cutoff functions
\(\chi,\chi^\sharp\), and \(d_\varepsilon\ge1/8\) give
\[
\|\rho_{1,\varepsilon}\|_{C^1([0,1])}
+\|\rho_{2,\varepsilon}\|_{C^1([0,1])}
\le C\varepsilon^{-2},
\]
where \(C>0\) is independent of \((\varepsilon,\eta)\).
The coefficient vectors are uniformly bounded. Since
\(G_\varepsilon\) is independent of \(\eta\),
\[
\partial_\eta\boldsymbol\beta_{c,\varepsilon}
=G_\varepsilon^{-1}\binom{0}{L}.
\]
This derivative is uniformly bounded as well. Since
\(\rho_{j,\varepsilon}\) is independent of \(\eta\),
\eqref{eq:kb-direct-definition}--\eqref{eq:kc-direct-definition}
and the uniform coefficient bounds give, after increasing \(C\) if
necessary, for \(0<\varepsilon\le1\),
\[
K_\varepsilon
\le1+2\delta_0+
C\bigl(\|\rho_{1,\varepsilon}\|_{C^1([0,1])}
+\|\rho_{2,\varepsilon}\|_{C^1([0,1])}\bigr)
\le C_{\mathrm{pr}}\varepsilon^{-2}.
\]
This proves \eqref{eq:profile-bound}.
\end{proof}

\section{Proof for Spatially Concentrated Prey}
\label{sec:concentrated-prey}

In this section, we construct equilibria with spatially concentrated prey and transfer the
matrix Hopf crossing to the full reaction--diffusion system, thereby proving
Theorem~\ref{thm:main}.

\subsection{Construction of the coefficient families}
\label{subsec:prey-coefficients}

Before analyzing stability, we first recover reaction coefficients that realize the
prescribed equilibrium and linearization while satisfying the required sign conditions.

Lemma~\ref{lem:positive-profiles} provides positive functions satisfying
\eqref{eq:moments} and \eqref{eq:profile-bound}.
Recall that \(\phi_{1,\varepsilon}\) is the positive
\(L^2((0,1))\)-normalized first eigenfunction of \(\cH_\varepsilon\)
in \eqref{eq:schrodinger}. For the linearization, define
\begin{equation}
A_\varepsilon
:=-\bigl(\cH_\varepsilon-\lambda_{1,\varepsilon}\mathrm{Id}\bigr).
\label{eq:Aepsilon}
\end{equation}
We seek reaction coefficients that realize the following linearization.

\begin{equation}
\cL_{\varepsilon,D,\eta}
=
\begin{pmatrix}
A_\varepsilon-\dfrac{\Lambda_{1,\varepsilon}}{2}\mathrm{Id}
&-\Lambda_{1,\varepsilon} k_{b,\varepsilon}\\[1mm]
\Lambda_{1,\varepsilon} k_{c,\varepsilon,\eta}
&D\partial_{xx}-\gamma_\varepsilon \Lambda_{1,\varepsilon}\mathrm{Id}
\end{pmatrix}.
\label{eq:target-linearization}
\end{equation}
Here \(0<\gamma_\varepsilon<\gamma_0\) is to be chosen, with
\(\gamma_0\) fixed by \eqref{eq:Lgamma-range}. The scalar functions \(k_{b,\varepsilon}\) and
\(k_{c,\varepsilon,\eta}\) act by pointwise multiplication.

\begin{lemma}[Coefficient families with spatially concentrated prey]
\label{lem:prey-coefficients}
For all sufficiently small \(\varepsilon\), there are constants
\(V_\varepsilon^\tau>0\), \(0<\gamma_\varepsilon^\tau<\gamma_0\), and
smooth coefficient families independent of \(D\), indexed by \(\tau=+,-\),
such that, with \(d=\varepsilon^2\),
\[
r_\varepsilon^\tau,a_\varepsilon,b_\varepsilon,
c_{\varepsilon,\eta}^\tau,e_\varepsilon^\tau>0
\qquad
m_{\varepsilon,\eta}^{+}>0>m_{\varepsilon,\eta}^{-}
\quad\text{on }[0,1]\times I.
\]
For every \(D>0\), \(\eta\in I\), and \(\tau\in\{+,-\}\), the unique
positive equilibrium is
\begin{equation}
(u_{\varepsilon,\eta}^{*,\tau},v_{\varepsilon,\eta}^{*,\tau})
=(\phi_{1,\varepsilon},V_\varepsilon^\tau).
\label{eq:exact-equilibrium}
\end{equation}
Its linearization is \eqref{eq:target-linearization}, with
\(\gamma_\varepsilon=\gamma_\varepsilon^\tau\).
\end{lemma}

\begin{proof}

Let \(V_\varepsilon>0\) and
\(0<\gamma_\varepsilon<\gamma_0\). We prescribe
\begin{equation}
(u_*,v_*)=(\phi_{1,\varepsilon},V_\varepsilon)
\label{eq:desired-equilibrium}
\end{equation}
as the exact equilibrium. The equilibrium identities are
\begin{equation}
r_\varepsilon-a_\varepsilon\phi_{1,\varepsilon}-b_\varepsilon V_\varepsilon
=\lambda_{1,\varepsilon}-W
\quad\text{and}\quad
m_{\varepsilon,\eta}+c_{\varepsilon,\eta}\phi_{1,\varepsilon}
-e_\varepsilon V_\varepsilon=0.
\label{eq:balance-identities}
\end{equation}
Substituting \eqref{eq:desired-equilibrium} into the general linearization
\eqref{eq:linearized-operator} and using the two identities in
\eqref{eq:balance-identities} yields
\begin{equation}
\begin{pmatrix}
A_\varepsilon-a_\varepsilon\phi_{1,\varepsilon}
&-b_\varepsilon\phi_{1,\varepsilon}\\
c_{\varepsilon,\eta}V_\varepsilon
&D\partial_{xx}-e_\varepsilon V_\varepsilon
\end{pmatrix}.
\label{eq:general-linearization}
\end{equation}

Comparing \eqref{eq:general-linearization} with
\eqref{eq:target-linearization} yields
\begin{subequations}
\label{eq:block-requirements}
\begin{align}
&a_\varepsilon\phi_{1,\varepsilon}=\Lambda_{1,\varepsilon}/2,
\qquad b_\varepsilon\phi_{1,\varepsilon}=\Lambda_{1,\varepsilon} k_{b,\varepsilon},
\label{eq:block-requirements-prey}\\
&c_{\varepsilon,\eta}V_\varepsilon=\Lambda_{1,\varepsilon} k_{c,\varepsilon,\eta},
\qquad
e_\varepsilon V_\varepsilon=\gamma_\varepsilon \Lambda_{1,\varepsilon}.
\label{eq:block-requirements-predator}
\end{align}
\end{subequations}
We then solve \eqref{eq:balance-identities} for
\(r_\varepsilon,m_{\varepsilon,\eta}\). The six coefficients are
\begin{subequations}
\label{eq:coefficients}
\begin{align}
a_\varepsilon(x)&:=\frac{\Lambda_{1,\varepsilon}}{2\phi_{1,\varepsilon}(x)},
\qquad
b_\varepsilon(x):=\frac{\Lambda_{1,\varepsilon} k_{b,\varepsilon}(x)}
{\phi_{1,\varepsilon}(x)},
\label{eq:coefficients-ab}\\
c_{\varepsilon,\eta}(x)&:=\frac{\Lambda_{1,\varepsilon}}{V_\varepsilon}
k_{c,\varepsilon,\eta}(x),
\qquad
e_\varepsilon(x):=\frac{\gamma_\varepsilon \Lambda_{1,\varepsilon}}
{V_\varepsilon},
\label{eq:coefficients-ce}\\
m_{\varepsilon,\eta}(x)
&:=\gamma_\varepsilon \Lambda_{1,\varepsilon}
-\frac{\Lambda_{1,\varepsilon}}{V_\varepsilon}
k_{c,\varepsilon,\eta}(x)\phi_{1,\varepsilon}(x),
\label{eq:coefficients-m}\\
r_\varepsilon(x)
&:=\lambda_{1,\varepsilon}-W(x)
+\frac{\Lambda_{1,\varepsilon}}{2}
+\frac{V_\varepsilon \Lambda_{1,\varepsilon} k_{b,\varepsilon}(x)}
{\phi_{1,\varepsilon}(x)}.
\label{eq:coefficients-r}
\end{align}
\end{subequations}
For \(t\in\R\), write \(t_+:=\max\{t,0\}\).
To control the sign of \(r_\varepsilon\), define
\begin{equation}
\Theta_\varepsilon^r
:=\left\|
\frac{(W(x)-\lambda_{1,\varepsilon}-\Lambda_{1,\varepsilon}/2)_+
\phi_{1,\varepsilon}(x)}
{\Lambda_{1,\varepsilon} k_{b,\varepsilon}(x)}
\right\|_{L^\infty((0,1))}.
\label{eq:prey-growth-threshold}
\end{equation}
Also set
\[
\underline\kappa_\varepsilon
:=\min_{(x,\eta)\in[0,1]\times I}
k_{c,\varepsilon,\eta}(x)\phi_{1,\varepsilon}(x),
\quad\text{and}\quad
\overline\kappa_\varepsilon
:=\max_{(x,\eta)\in[0,1]\times I}
k_{c,\varepsilon,\eta}(x)\phi_{1,\varepsilon}(x).
\]
By compactness and strict positivity, we have
\(0<\underline\kappa_\varepsilon\le\overline\kappa_\varepsilon<\infty\).

We first construct the family with negative predator growth. Choose
\begin{equation}
V_\varepsilon^->1+\Theta_\varepsilon^r
\quad\text{and}\quad
0<\gamma_\varepsilon^-<
\min\left\{\gamma_0,
\frac{\underline\kappa_\varepsilon}{2V_\varepsilon^-}\right\}.
\label{eq:negative-choice}
\end{equation}
Then, for every \((x,\eta)\in[0,1]\times I\),
\begin{equation}
m_{\varepsilon,\eta}^-(x)
\le-\frac{\Lambda_{1,\varepsilon}\underline\kappa_\varepsilon}
{2V_\varepsilon^-}<0.
\label{eq:m-negative}
\end{equation}

For the family with positive predator growth, choose
\begin{equation}
V_\varepsilon^+>
\max\left\{1+\Theta_\varepsilon^r,
\frac{\overline\kappa_\varepsilon}{\gamma_0}\right\}
\quad\text{and}\quad
\frac{\overline\kappa_\varepsilon}{V_\varepsilon^+}
<\gamma_\varepsilon^+<\gamma_0.
\label{eq:positive-choice}
\end{equation}
The bound
\(V_\varepsilon^+>\overline\kappa_\varepsilon/\gamma_0\) makes the interval
for \(\gamma_\varepsilon^+\) in \eqref{eq:positive-choice} nonempty. The
lower bound on \(\gamma_\varepsilon^+\), together with
\eqref{eq:coefficients-m}, gives
\begin{equation}
m_{\varepsilon,\eta}^+(x)
\ge \Lambda_{1,\varepsilon}\left(
\gamma_\varepsilon^+-
\frac{\overline\kappa_\varepsilon}{V_\varepsilon^+}
\right)>0.
\label{eq:m-positive}
\end{equation}
For the choices \eqref{eq:negative-choice} and \eqref{eq:positive-choice},
\eqref{eq:prey-growth-threshold} gives
\[
 r_\varepsilon^\tau(x)
 \ge\bigl(V_\varepsilon^\tau-\Theta_\varepsilon^r\bigr)
 \frac{\Lambda_{1,\varepsilon} k_{b,\varepsilon}(x)}
      {\phi_{1,\varepsilon}(x)}
 >0,\qquad x\in[0,1].
\]
Thus, both sets of coefficients define valid predator--prey
systems with strictly positive interaction coefficients.

For \(\tau\in\{+,-\}\), substitute
\((V_\varepsilon^\tau,\gamma_\varepsilon^\tau)\) into
\eqref{eq:coefficients}. The balance identities
\eqref{eq:balance-identities} give \eqref{eq:exact-equilibrium}. By
\cite[Theorem~4.1]{LopezGomezPardo1998}, the equilibrium in
\eqref{eq:exact-equilibrium} is the unique positive equilibrium for every
\(D>0\) with \(d=\varepsilon^2\). Its explicit formula is independent
of \(\eta\). Substitution into \eqref{eq:general-linearization} gives
\eqref{eq:target-linearization}.
\end{proof}

\subsection{Spectral reduction of the full linearization}
\label{subsec:prey-reduction}

In this subsection, we reduce the eigenvalue problem in a half-plane containing the imaginary axis
to an exact matrix equation with controlled correction terms.

On \(L^2((0,1))^2\), the full linearization is a bounded perturbation
of the diagonal Neumann diffusion operator and has compact resolvent.
To identify the eigenvalue crossing, fix either family and suppress the label
\(\tau\) throughout the remainder of this section. All subsequent estimates
are uniform for
\(0\le\gamma_\varepsilon\le\gamma_0\). Let \(P_\varepsilon\) be the
orthogonal projection onto
\(\operatorname{span}_{\C}\{\zeta_{1,\varepsilon},\zeta_{2,\varepsilon}\}\), and let \(P\)
be the orthogonal projection onto the constant functions. Set
\begin{equation}
\cP_\varepsilon=\operatorname{diag}(P_\varepsilon,P)
\quad\text{and}\quad
\cQ_\varepsilon=\mathrm{Id}-\cP_\varepsilon.
\label{eq:PQ}
\end{equation}
Thus \(\cP_\varepsilon\) keeps the three selected functions, whereas
\(\cQ_\varepsilon\) keeps precisely the remaining functions:
\begin{equation}
\operatorname{Ran}\cQ_\varepsilon
=\left\{(\xi,\psi)\in L^2((0,1))^2:
\xi\perp\operatorname{span}_{\C}
\{\zeta_{1,\varepsilon},\zeta_{2,\varepsilon}\},\quad
\langle\psi,\1\rangle=0\right\}.
\label{eq:Q-range}
\end{equation}
Define the ordered orthonormal basis by
\begin{equation}
\mathcal B_\varepsilon
:=\bigl((\zeta_{1,\varepsilon},0),
(\zeta_{2,\varepsilon},0),(0,\1)\bigr).
\label{eq:slow-basis}
\end{equation}

Hereafter, \(\|\cdot\|\) denotes the operator norm induced by the indicated
\(L^2\) norms on the relevant subspaces, or by the Euclidean norm for
matrices on \(\C^3\).
To enclose the eigenvalues of the target matrices in a common disk, choose
\begin{equation}
R:=1+2\sup_{(\gamma,\eta)\in[0,\gamma_0]\times I}
\|M_\gamma(\eta)\|.
\label{eq:spectral-window-radius}
\end{equation}
Then every eigenvalue \(z\) of every matrix \(M_\gamma(\eta)\) satisfies
\(|z|\le\|M_\gamma(\eta)\|\le(R-1)/2<R/2\).
With \(K_\varepsilon\) from \eqref{eq:profile-bound}, set
\begin{equation}
\delta_{\varepsilon,D}
:=\Lambda_{1,\varepsilon} K_\varepsilon^2
\left(\frac1{\Lambda_{2,\varepsilon}}+\frac1D\right).
\label{eq:delta}
\end{equation}
For \(\lambda=\Lambda_{1,\varepsilon}z\), define on the space in \eqref{eq:Q-range}
\begin{equation}
B(z,\eta):=
\cQ_\varepsilon(\Lambda_{1,\varepsilon} z-\cL_{\varepsilon,D,\eta})
\cQ_\varepsilon\quad\text{with}\quad
\operatorname{Dom}(B)=\operatorname{Dom}(\cL_{\varepsilon,D,\eta})
\cap\operatorname{Ran}\cQ_\varepsilon.
\label{eq:B-definition}
\end{equation}

\begin{lemma}[Invertibility on the complementary space]
\label{lem:complement-inverse}
There is \(\varepsilon_{\mathrm{red}}>0\) such that, if
\(0<\varepsilon<\varepsilon_{\mathrm{red}}\), \(D>0\), and
\(\delta_{\varepsilon,D}\le\min\{1/4,\pi^2/(2R)\}\), then, uniformly for
\((z,\eta)\in\C\times I\) with \(\operatorname{Re} z\ge-R\),
\begin{equation}
\|B(z,\eta)^{-1}\|
\le4\left(\frac1{\Lambda_{2,\varepsilon}}+\frac1D\right).
\label{eq:high-resolvent}
\end{equation}
\end{lemma}

\begin{proof}
For \((\xi,\psi)\in\operatorname{Dom}(B)\),
\eqref{eq:Q-range} gives \(\xi\perp\phi_{1,\varepsilon},\phi_{2,\varepsilon}\).
The variational characterization of \(\lambda_{3,\varepsilon}\) yields
\[
\langle\cH_\varepsilon\xi,\xi\rangle
\ge\lambda_{3,\varepsilon}\|\xi\|_{L^2((0,1))}^2.
\]
Hence, by \eqref{eq:Aepsilon}, we obtain
\begin{equation}
\left\langle
-\left(A_\varepsilon-\frac{\Lambda_{1,\varepsilon}}{2}\mathrm{Id}\right)\xi,\xi
\right\rangle
\ge
\left(\Lambda_{2,\varepsilon}+\frac{\Lambda_{1,\varepsilon}}{2}\right)\|\xi\|_{L^2((0,1))}^2.
\label{eq:complementary-prey}
\end{equation}
The mean-zero condition in \eqref{eq:Q-range} and the Neumann Poincar\'e
inequality similarly give
\begin{equation}
\left\langle
-\left(D\partial_{xx}-\gamma_\varepsilon \Lambda_{1,\varepsilon}\mathrm{Id}\right)
\psi,\psi\right\rangle
=D\|\psi'\|_{L^2((0,1))}^2
+\gamma_\varepsilon \Lambda_{1,\varepsilon}\|\psi\|_{L^2((0,1))}^2
\ge
\left(D\pi^2+\gamma_\varepsilon \Lambda_{1,\varepsilon}\right)\|\psi\|_{L^2((0,1))}^2.
\label{eq:complementary-predator}
\end{equation}
By the upper bound for \(\Lambda_{1,\varepsilon}\) and the lower bound for
\(\Lambda_{2,\varepsilon}\) in \eqref{eq:scale-separation}, we obtain
\[
\frac{R\Lambda_{1,\varepsilon}}{\Lambda_{2,\varepsilon}}
 \le\frac{8RC_1}{\nu_2-\nu_1}
 \varepsilon^{-1}e^{-\kappa_1/\varepsilon}
 \longrightarrow0.
\]
Moreover, \(K_\varepsilon\ge1\) and \eqref{eq:delta} give
\[
\frac{\Lambda_{1,\varepsilon}}D
\le\frac{\Lambda_{1,\varepsilon}K_\varepsilon^2}D
\le\delta_{\varepsilon,D}\le\frac{\pi^2}{2R}.
\]
Hence, for small \(\varepsilon\),
\begin{equation}
R\Lambda_{1,\varepsilon}\le\frac{\Lambda_{2,\varepsilon}}2,\qquad
R\Lambda_{1,\varepsilon}\le\frac{D\pi^2}2.
\label{eq:complement-shift-bounds}
\end{equation}

Fix \((z,\eta)\) with \(\operatorname{Re} z\ge-R\), and write \(B=B(z,\eta)\).
For \(w=(\xi,\psi)\in\operatorname{Dom}(B)\), we have
\(\cQ_\varepsilon w=w\).
Applying the Cauchy--Schwarz inequality to the two off-diagonal terms in
\eqref{eq:target-linearization}, then using
\eqref{eq:complementary-prey}--\eqref{eq:complement-shift-bounds}
and \eqref{eq:profile-bound}, gives
\begin{equation}
\begin{aligned}
\operatorname{Re}\langle Bw,w\rangle
&\ge
\frac{\Lambda_{2,\varepsilon}}2\|\xi\|_{L^2((0,1))}^2
+\frac{D\pi^2}2\|\psi\|_{L^2((0,1))}^2
-\Lambda_{1,\varepsilon}K_\varepsilon\|w\|_{L^2((0,1))^2}^2\\
&\ge
\left(\frac12-\frac{\delta_{\varepsilon,D}}{K_\varepsilon}\right)
\left(\frac1{\Lambda_{2,\varepsilon}}+\frac1D\right)^{-1}\|w\|_{L^2((0,1))^2}^2\\
&\ge\frac14
\left(\frac1{\Lambda_{2,\varepsilon}}+\frac1D\right)^{-1}\|w\|_{L^2((0,1))^2}^2.
\end{aligned}
\label{eq:complement-real-part-bound}
\end{equation}
Here we used \eqref{eq:delta}, \(K_\varepsilon\ge1\), and
\(\delta_{\varepsilon,D}\le1/4\).
Another application of the Cauchy--Schwarz inequality yields
\begin{equation}
\|w\|_{L^2((0,1))^2}
\le4\left(\frac1{\Lambda_{2,\varepsilon}}+\frac1D\right)\|Bw\|_{L^2((0,1))^2}.
\label{eq:complement-apriori-bound}
\end{equation}
The diagonal part of \eqref{eq:target-linearization} is self-adjoint
and commutes with \(\cQ_\varepsilon\), so its restriction to
\(\operatorname{Ran}\cQ_\varepsilon\) is self-adjoint.
The remaining terms in \(B\) are bounded. Hence \(B\) is closed and
densely defined, with
\[
\operatorname{Dom}(B^*)=\operatorname{Dom}(B).
\]
Estimate \eqref{eq:complement-apriori-bound} gives injectivity.
It also gives closed range: if \(Bw_n\) converges, then \(w_n\) is Cauchy,
so closedness of \(B\) places the limit of \(Bw_n\) in its range.
For \(w\in\operatorname{Dom}(B^*)\),
\[
\operatorname{Re}\langle B^*w,w\rangle=\operatorname{Re}\langle Bw,w\rangle.
\]
Thus \eqref{eq:complement-real-part-bound} gives
\[
(\operatorname{Ran}B)^\perp=\ker B^*=\{0\}.
\]
The range of \(B\) is both closed and dense in
\(\operatorname{Ran}\cQ_\varepsilon\), hence equals that space.
Therefore \(B\) is invertible, with \eqref{eq:high-resolvent}
following from \eqref{eq:complement-apriori-bound}.
\end{proof}

Using Lemma~\ref{lem:complement-inverse}, define the scaled Schur complement
\begin{equation}
\mathfrak F_{\varepsilon,D}(z,\eta)
:=\frac1{\Lambda_{1,\varepsilon}}\left[
\begin{aligned}
&\cP_\varepsilon(\Lambda_{1,\varepsilon} z-\cL_{\varepsilon,D,\eta})
\cP_\varepsilon\\[-1mm]
&\quad-
\cP_\varepsilon\cL_{\varepsilon,D,\eta}\cQ_\varepsilon
B(z,\eta)^{-1}
\cQ_\varepsilon\cL_{\varepsilon,D,\eta}\cP_\varepsilon
\end{aligned}
\right]_{\mathcal B_\varepsilon}.
\label{eq:reduced-matrix}
\end{equation}

\Needspace{4\baselineskip}
\begin{lemma}[Schur reduction and spectral equivalence]
\label{lem:three-mode-reduction}
Under the hypotheses of Lemma~\ref{lem:complement-inverse}, for every
\((z,\eta)\in\C\times I\) with \(\operatorname{Re} z\ge-R\),
\(\Lambda_{1,\varepsilon}z\) is an eigenvalue of
\(\cL_{\varepsilon,D,\eta}\) if and only if
\begin{equation}
\det\mathfrak F_{\varepsilon,D}(z,\eta)=0.
\label{eq:spectral-equivalence}
\end{equation}
The algebraic multiplicity of \(\Lambda_{1,\varepsilon} z\) as an eigenvalue equals
the multiplicity of \(z\) as a zero of
\(\det\mathfrak F_{\varepsilon,D}(\,\cdot\,,\eta)\).
\end{lemma}

\begin{proof}

Identify \(\cX_\varepsilon\) with \(\C^3\) through the ordered orthonormal
basis \(\mathcal B_\varepsilon\) in \eqref{eq:slow-basis}. Put
\(\lambda=\Lambda_{1,\varepsilon} z\). For
\(w\in\operatorname{Dom}(\cL_{\varepsilon,D,\eta})\), write
\[
w=p+q,\qquad p=\cP_\varepsilon w,
\qquad\text{and}\qquad q=\cQ_\varepsilon w.
\]
The selected functions \(\zeta_{1,\varepsilon}\),
\(\zeta_{2,\varepsilon}\), and \(\1\) belong to \(H_N^2((0,1))\), so both
projections preserve the operator domain.
Applying the complementary projections in \eqref{eq:PQ} to
\((\Lambda_{1,\varepsilon}z-\cL_{\varepsilon,D,\eta})w=0\) gives
\begin{equation}
\left\{
\begin{aligned}
&\cP_\varepsilon(\Lambda_{1,\varepsilon} z-\cL_{\varepsilon,D,\eta})
\cP_\varepsilon p
-\cP_\varepsilon\cL_{\varepsilon,D,\eta}\cQ_\varepsilon q=0,\\
&-\cQ_\varepsilon\cL_{\varepsilon,D,\eta}\cP_\varepsilon p
+B(z,\eta)q=0.
\end{aligned}
\right.
\label{eq:pq-system}
\end{equation}
Conversely, adding the two equations in \eqref{eq:pq-system} recovers
the original equation.
The second equation in \eqref{eq:pq-system} gives
\begin{equation}
q=B(z,\eta)^{-1}
\cQ_\varepsilon\cL_{\varepsilon,D,\eta}\cP_\varepsilon p.
\label{eq:remaining-component}
\end{equation}
Substitution of \eqref{eq:remaining-component} into the first equation in
\eqref{eq:pq-system} gives
\(\mathfrak F_{\varepsilon,D}(z,\eta)p=0\). Thus
\[
\begin{aligned}
\ker\mathfrak F_{\varepsilon,D}(z,\eta)
&\longrightarrow\ker(\Lambda_{1,\varepsilon}z-\cL_{\varepsilon,D,\eta}),\\
p&\longmapsto w=p+B(z,\eta)^{-1}
\cQ_\varepsilon\cL_{\varepsilon,D,\eta}\cP_\varepsilon p
\end{aligned}
\]
has inverse \(w\mapsto\cP_\varepsilon w=p\).
Hence \eqref{eq:spectral-equivalence} follows.

To verify the statement about algebraic multiplicity, write, relative to
\(\cX_\varepsilon\oplus\operatorname{Ran}\cQ_\varepsilon\),
\[
\Lambda_{1,\varepsilon} z-\cL_{\varepsilon,D,\eta}
=\begin{pmatrix}
A_{11}&A_{12}\\
A_{21}&B
\end{pmatrix}.
\]
Direct multiplication gives
\begin{equation}
\begin{pmatrix}\mathrm{Id}&-A_{12}B^{-1}\\0&\mathrm{Id}\end{pmatrix}
\begin{pmatrix}A_{11}&A_{12}\\A_{21}&B\end{pmatrix}
\begin{pmatrix}\mathrm{Id}&0\\-B^{-1}A_{21}&\mathrm{Id}\end{pmatrix}
=
\begin{pmatrix}\Lambda_{1,\varepsilon}\mathfrak F_{\varepsilon,D}&0\\0&B\end{pmatrix}.
\label{eq:block-factorization}
\end{equation}
Fix \(\varepsilon,D,\eta\). We use the spaces
\begin{equation}
\begin{aligned}
\mathscr H&:=L^2((0,1))^2
=\cX_\varepsilon\oplus\operatorname{Ran}\cQ_\varepsilon,\\
\mathscr D&:=\operatorname{Dom}(\cL_{\varepsilon,D,\eta})=\{w\in H^2((0,1))^2:w'(0)=w'(1)=0\}.
\end{aligned}
\label{eq:reduction-domain-spaces}
\end{equation}
The space \(\mathscr D\) is independent of \(z\). On \(\mathscr D\), take
the graph norm, defined explicitly by
\[
\|w\|_{\mathscr D}:=
\left(\|w\|_{L^2((0,1))^2}^2+
\|\cL_{\varepsilon,D,\eta}w\|_{L^2((0,1))^2}^2\right)^{1/2}.
\]
Thus \(\cL_{\varepsilon,D,\eta}:\mathscr D\to\mathscr H\) is bounded.
For \(\operatorname{Re} z_0\ge-R\), we have
\begin{equation}
B(z,\eta)B(z_0,\eta)^{-1}
=\mathrm{Id}+\Lambda_{1,\varepsilon}(z-z_0)B(z_0,\eta)^{-1}.
\label{eq:normalized-complement}
\end{equation}
The operator in \eqref{eq:normalized-complement} has an analytic inverse
near \(z_0\), including \(\operatorname{Re} z_0=-R\).
In \eqref{eq:block-factorization}, denote the triangular factors by
\(S(z)\) on the left and \(T(z)\) on the right.
The maps \(S(z):\mathscr H\to\mathscr H\) and
\(T(z):\mathscr D\to\mathscr D\) depend analytically on \(z\).
Both maps have bounded inverses near \(z_0\).

By Lemma~\ref{lem:complement-inverse}, with the domain in
\eqref{eq:B-definition},
\(B(z_0,\eta)^{-1}\) is a bounded map from
\(\operatorname{Ran}\cQ_\varepsilon\) onto
\(\mathscr D\cap\operatorname{Ran}\cQ_\varepsilon\).
Since
\(\mathscr D=\cX_\varepsilon\oplus
(\mathscr D\cap\operatorname{Ran}\cQ_\varepsilon)\), the map
\begin{equation}
\begin{aligned}
J_0&:=\operatorname{diag}(I_3,B(z_0,\eta)^{-1}):
\mathscr H\longrightarrow\mathscr D,\\
J_0(p+f)&=p+B(z_0,\eta)^{-1}f,
\qquad p\in\cX_\varepsilon,\quad
f\in\operatorname{Ran}\cQ_\varepsilon,
\end{aligned}
\label{eq:domain-identification}
\end{equation}
is bounded and bijective, with bounded inverse
\[
J_0^{-1}w=\cP_\varepsilon w+B(z_0,\eta)\cQ_\varepsilon w,
\qquad w\in\mathscr D.
\]
Here \(I_3\) acts on \(\cX_\varepsilon\) in the basis
\eqref{eq:slow-basis}.
After composing \eqref{eq:block-factorization} on the right with \(J_0\),
the middle operator becomes
\((\Lambda_{1,\varepsilon}z-\cL_{\varepsilon,D,\eta})J_0\);
the right triangular factor becomes \(J_0^{-1}T(z)J_0\).
These are bounded operators on \(\mathscr H\).
The identity therefore gives an equivalence of bounded analytic Fredholm
families of index zero on \(\mathscr H\), whose lower block
\eqref{eq:normalized-complement} remains invertible.
Since \(\cL_{\varepsilon,D,\eta}\) has compact resolvent, its eigenvalues
are isolated and have finite algebraic multiplicity.
Let \(m\ge1\) be the order of the zero \(z_0\) of
\(\det\mathfrak F_{\varepsilon,D}(\,\cdot\,,\eta)\). Thus
\begin{equation}
\det\mathfrak F_{\varepsilon,D}(z,\eta)
=(z-z_0)^m h(z),\qquad h(z_0)\ne0,
\label{eq:determinant-zero-factorization}
\end{equation}
where \(m\) is an integer and \(h\) is analytic near \(z_0\).
Since \(B(z,\eta)\) is invertible, \eqref{eq:block-factorization} together
with \cite[\S\S~1.2--1.3 and Theorem~2.1]{GohbergSigal1971} shows that the
eigenvalue \(\Lambda_{1,\varepsilon}z_0\) has algebraic multiplicity \(m\).
In particular, a simple zero of the determinant gives an algebraically
simple eigenvalue.
\end{proof}

\begin{lemma}[Approximation by the target matrix]
\label{lem:reduced-matrix-approximation}
Under the hypotheses of Lemma~\ref{lem:complement-inverse},
\(\mathfrak F_{\varepsilon,D}(z,\eta)\) is analytic in \(z\) and \(C^1\) in
\(\eta\). Uniformly for \(\operatorname{Re} z\ge-R\) and \(\eta\in I\), it satisfies
\begin{subequations}
\label{eq:reduced-matrix-approximation}
\begin{align}
\|\mathfrak F_{\varepsilon,D}
 -(zI_3-M_{\gamma_\varepsilon}(\eta))\|
&\le4\delta_{\varepsilon,D},
\label{eq:reduced-matrix-error-zero}\\
\|\partial_z\mathfrak F_{\varepsilon,D}-I_3\|
&\le4\delta_{\varepsilon,D},
\label{eq:reduced-matrix-error-z}\\
\|\partial_\eta\mathfrak F_{\varepsilon,D}
 +\partial_\eta M_{\gamma_\varepsilon}(\eta)\|
&\le12\delta_{\varepsilon,D}.
\label{eq:reduced-matrix-error-eta}
\end{align}
\end{subequations}
\end{lemma}

\begin{proof}
By \eqref{eq:Aepsilon}, the first two eigenfunctions satisfy
\begin{equation}
A_\varepsilon\phi_{1,\varepsilon}=0
\quad\text{and}\quad
A_\varepsilon\phi_{2,\varepsilon}
=-\Lambda_{1,\varepsilon}\phi_{2,\varepsilon}.
\label{eq:eigenpairs}
\end{equation}
Combining \eqref{eq:eigenpairs} with \eqref{eq:localized-modes} gives
\begin{equation}
\begin{aligned}
\left(A_\varepsilon-\frac{\Lambda_{1,\varepsilon}}{2}\mathrm{Id}\right)
\zeta_{1,\varepsilon}
&=-\Lambda_{1,\varepsilon}\zeta_{1,\varepsilon}
+\frac{\Lambda_{1,\varepsilon}}{2}\zeta_{2,\varepsilon},\\
\left(A_\varepsilon-\frac{\Lambda_{1,\varepsilon}}{2}\mathrm{Id}\right)
\zeta_{2,\varepsilon}
&=\frac{\Lambda_{1,\varepsilon}}{2}\zeta_{1,\varepsilon}
-\Lambda_{1,\varepsilon}\zeta_{2,\varepsilon}.
\end{aligned}
\label{eq:localized-operator-action}
\end{equation}
Since \(\|\1\|_{L^2((0,1))}=1\), the two moment pairs in
\eqref{eq:moments} give the upper-right column \(-(L,1)^\top\) and the
lower-left row \((1,\eta L)\), respectively.
Finally,
\begin{equation}
(D\partial_{xx}-\gamma_\varepsilon \Lambda_{1,\varepsilon})\1
=-\gamma_\varepsilon \Lambda_{1,\varepsilon}\1.
\label{eq:constant-operator-action}
\end{equation}
Using \eqref{eq:localized-operator-action}, \eqref{eq:constant-operator-action}, and
\eqref{eq:moments} in \eqref{eq:target-linearization}, then dividing the
compressed operator by \(\Lambda_{1,\varepsilon}\), yields, for every
\((D,\eta)\in(0,\infty)\times I\),
\begin{equation}
\left[
\left.\frac1{\Lambda_{1,\varepsilon}}
\cP_\varepsilon\cL_{\varepsilon,D,\eta}\cP_\varepsilon
\right|_{\cX_\varepsilon}
\right]_{\mathcal B_\varepsilon}
=M_{\gamma_\varepsilon}(\eta).
\label{eq:compression}
\end{equation}

By \eqref{eq:compression}, the first term in \eqref{eq:reduced-matrix} is
\(zI_3-M_{\gamma_\varepsilon}(\eta)\). Therefore we obtain
\begin{equation}
\mathfrak F_{\varepsilon,D}(z,\eta)
-(zI_3-M_{\gamma_\varepsilon}(\eta))
=-\frac1{\Lambda_{1,\varepsilon}}
\left[
\cP_\varepsilon\cL_{\varepsilon,D,\eta}\cQ_\varepsilon
B(z,\eta)^{-1}
\cQ_\varepsilon\cL_{\varepsilon,D,\eta}\cP_\varepsilon
\right]_{\mathcal B_\varepsilon}.
\label{eq:reduced-correction}
\end{equation}
The diagonal operators in \eqref{eq:target-linearization} preserve both
\(\operatorname{Ran}\cP_\varepsilon\) and
\(\operatorname{Ran}\cQ_\varepsilon\). Hence the two off-diagonal
cross-species terms in \eqref{eq:target-linearization} are the only
contributions to the operators coupling these subspaces. More explicitly,
\begin{subequations}
\label{eq:cross-space-blocks}
\begin{align}
\cP_\varepsilon\cL_{\varepsilon,D,\eta}\cQ_\varepsilon
&=\Lambda_{1,\varepsilon}
\begin{pmatrix}
0&-P_\varepsilon k_{b,\varepsilon}(\mathrm{Id}-P)\\
P k_{c,\varepsilon,\eta}(\mathrm{Id}-P_\varepsilon)&0
\end{pmatrix},
\label{eq:cross-space-block-PQ}\\
\cQ_\varepsilon\cL_{\varepsilon,D,\eta}\cP_\varepsilon
&=\Lambda_{1,\varepsilon}
\begin{pmatrix}
0&-(\mathrm{Id}-P_\varepsilon)k_{b,\varepsilon}P\\
(\mathrm{Id}-P)k_{c,\varepsilon,\eta}P_\varepsilon&0
\end{pmatrix}.
\label{eq:cross-space-block-QP}
\end{align}
\end{subequations}
For \(w=(\xi,\psi)\), the two identities in
\eqref{eq:cross-space-blocks} give
\begin{equation}
\begin{aligned}
\|\cP_\varepsilon\cL_{\varepsilon,D,\eta}
       \cQ_\varepsilon w\|_{L^2((0,1))^2}^2
={}&\Lambda_{1,\varepsilon}^2
\bigl(
\|P_\varepsilon[k_{b,\varepsilon}(\mathrm{Id}-P)\psi]\|_{L^2((0,1))}^2\\
&\qquad+\|P[k_{c,\varepsilon,\eta}(\mathrm{Id}-P_\varepsilon)\xi]\|_{L^2((0,1))}^2
\bigr),\\
\|\cQ_\varepsilon\cL_{\varepsilon,D,\eta}
       \cP_\varepsilon w\|_{L^2((0,1))^2}^2
={}&\Lambda_{1,\varepsilon}^2
\bigl(
\|(\mathrm{Id}-P_\varepsilon)[k_{b,\varepsilon}P\psi]\|_{L^2((0,1))}^2\\
&\qquad+\|(\mathrm{Id}-P)[k_{c,\varepsilon,\eta}P_\varepsilon\xi]\|_{L^2((0,1))}^2
\bigr).
\end{aligned}
\label{eq:cross-space-applied-norms}
\end{equation}
By \eqref{eq:profile-bound}, we bound both right-hand sides in
\eqref{eq:cross-space-applied-norms} by
\[
\Lambda_{1,\varepsilon}^2
\left(
\|k_{b,\varepsilon}\|_{L^\infty((0,1))}^2\|\psi\|_{L^2((0,1))}^2
+\|k_{c,\varepsilon,\eta}\|_{L^\infty((0,1))}^2\|\xi\|_{L^2((0,1))}^2
\right)
\le \Lambda_{1,\varepsilon}^2K_\varepsilon^2\|w\|_{L^2((0,1))^2}^2.
\]
Taking the supremum over \(\|w\|_{L^2((0,1))^2}=1\) yields
\begin{equation}
\|\cP_\varepsilon\cL_{\varepsilon,D,\eta}\cQ_\varepsilon\|
\le \Lambda_{1,\varepsilon} K_\varepsilon,
\qquad
\|\cQ_\varepsilon\cL_{\varepsilon,D,\eta}\cP_\varepsilon\|
\le \Lambda_{1,\varepsilon} K_\varepsilon.
\label{eq:cross-space-block-bounds}
\end{equation}
Applying \eqref{eq:cross-space-block-bounds} and
\eqref{eq:high-resolvent} to \eqref{eq:reduced-correction} gives
\[
\|\mathfrak F_{\varepsilon,D}
 -(zI_3-M_{\gamma_\varepsilon}(\eta))\|
\le
\frac1{\Lambda_{1,\varepsilon}}
 (\Lambda_{1,\varepsilon} K_\varepsilon)\cdot
 4\left(\frac1{\Lambda_{2,\varepsilon}}+\frac1D\right)\cdot
 (\Lambda_{1,\varepsilon} K_\varepsilon)
 =4\delta_{\varepsilon,D}.
\]
This proves \eqref{eq:reduced-matrix-error-zero}.

The operator \(B(z,\eta)\) is analytic in \(z\), \(C^1\) in \(\eta\), and
invertible by Lemma~\ref{lem:complement-inverse}. Hence its inverse has the
same regularity. We have
\(\partial_zB=\Lambda_{1,\varepsilon}
\mathrm{Id}_{\operatorname{Ran}\cQ_\varepsilon}\) and
\(\partial_\eta B=-\cQ_\varepsilon
(\partial_\eta\cL_{\varepsilon,D,\eta})\cQ_\varepsilon\).
Differentiating \(B(z,\eta)B(z,\eta)^{-1}=\mathrm{Id}\) therefore gives
\begin{equation}
\partial_zB^{-1}=-\Lambda_{1,\varepsilon} B^{-2}
\quad\text{and}\quad
\partial_\eta B^{-1}
=B^{-1}\cQ_\varepsilon
(\partial_\eta\cL_{\varepsilon,D,\eta})
\cQ_\varepsilon B^{-1}.
\label{eq:B-inverse-derivatives}
\end{equation}

The two cross-space operators in
\eqref{eq:cross-space-blocks} do not depend on \(z\). Therefore,
differentiating \eqref{eq:reduced-correction} and using the first identity in
\eqref{eq:B-inverse-derivatives} gives the exact formula
\[
\partial_z\mathfrak F_{\varepsilon,D}-I_3
=
\left[
\cP_\varepsilon\cL_{\varepsilon,D,\eta}\cQ_\varepsilon
B(z,\eta)^{-2}
\cQ_\varepsilon\cL_{\varepsilon,D,\eta}\cP_\varepsilon
\right]_{\mathcal B_\varepsilon}.
\]
Since \(\mathcal B_\varepsilon\) is orthonormal,
\eqref{eq:cross-space-block-bounds} and \eqref{eq:high-resolvent} give
\[
\begin{aligned}
\|\partial_z\mathfrak F_{\varepsilon,D}-I_3\|
&\le
(\Lambda_{1,\varepsilon} K_\varepsilon)
\,16\left(\frac1{\Lambda_{2,\varepsilon}}+\frac1D\right)^2
(\Lambda_{1,\varepsilon} K_\varepsilon)\\
&=16\frac{\delta_{\varepsilon,D}^2}{K_\varepsilon^2}
\le16\delta_{\varepsilon,D}^2
\le4\delta_{\varepsilon,D}.
\end{aligned}
\]
Here we used \eqref{eq:delta}, \(K_\varepsilon\ge1\), and
\(\delta_{\varepsilon,D}\le1/4\), proving
\eqref{eq:reduced-matrix-error-z}.

For the \(\eta\)-derivative, the quantities
\(\Lambda_{1,\varepsilon},\gamma_\varepsilon,P_\varepsilon,P\), and
\(k_{b,\varepsilon}\) are fixed; only \(k_{c,\varepsilon,\eta}\) varies.
Thus
\begin{equation}
\partial_\eta\cL_{\varepsilon,D,\eta}
=\Lambda_{1,\varepsilon}
\begin{pmatrix}
0&0\\
\partial_\eta k_{c,\varepsilon,\eta}&0
\end{pmatrix}.
\label{eq:operator-eta-derivative}
\end{equation}
Since the projections are independent of \(\eta\),
differentiating the two blocks in \eqref{eq:cross-space-blocks} and
using \eqref{eq:profile-bound} gives
\begin{equation}
\begin{aligned}
\|\partial_\eta(
\cP_\varepsilon\cL_{\varepsilon,D,\eta}\cQ_\varepsilon)\|
&\le \Lambda_{1,\varepsilon} K_\varepsilon,\\
\|\partial_\eta(
\cQ_\varepsilon\cL_{\varepsilon,D,\eta}\cP_\varepsilon)\|
&\le \Lambda_{1,\varepsilon} K_\varepsilon,\\
\|\cQ_\varepsilon
(\partial_\eta\cL_{\varepsilon,D,\eta})\cQ_\varepsilon\|
&\le \Lambda_{1,\varepsilon} K_\varepsilon.
\end{aligned}
\label{eq:cross-space-eta-bounds}
\end{equation}
The second identity in \eqref{eq:B-inverse-derivatives}, the third estimate
in \eqref{eq:cross-space-eta-bounds}, and \eqref{eq:high-resolvent} give
\begin{equation}
\|\partial_\eta B^{-1}\|
\le16\Lambda_{1,\varepsilon} K_\varepsilon
\left(\frac1{\Lambda_{2,\varepsilon}}+\frac1D\right)^2.
\label{eq:B-inverse-eta-bound}
\end{equation}

Differentiating all three factors in the correction
\eqref{eq:reduced-correction} gives
\begin{equation}
\begin{aligned}
\partial_\eta\mathfrak F_{\varepsilon,D}
+\partial_\eta M_{\gamma_\varepsilon}(\eta)
=-\frac1{\Lambda_{1,\varepsilon}}\bigl[&
\partial_\eta(
\cP_\varepsilon\cL_{\varepsilon,D,\eta}\cQ_\varepsilon)
B^{-1}\cQ_\varepsilon\cL_{\varepsilon,D,\eta}\cP_\varepsilon\\
&+\cP_\varepsilon\cL_{\varepsilon,D,\eta}\cQ_\varepsilon
(\partial_\eta B^{-1})
\cQ_\varepsilon\cL_{\varepsilon,D,\eta}\cP_\varepsilon\\
&+\cP_\varepsilon\cL_{\varepsilon,D,\eta}\cQ_\varepsilon B^{-1}
\partial_\eta(
\cQ_\varepsilon\cL_{\varepsilon,D,\eta}\cP_\varepsilon)
\bigr]_{\mathcal B_\varepsilon}.
\end{aligned}
\label{eq:reduced-eta-derivative}
\end{equation}
Applying \eqref{eq:cross-space-block-bounds},
\eqref{eq:cross-space-eta-bounds}, \eqref{eq:high-resolvent}, and
\eqref{eq:B-inverse-eta-bound} to the three terms in
\eqref{eq:reduced-eta-derivative} gives
\[
\begin{aligned}
\|\partial_\eta\mathfrak F_{\varepsilon,D}
+\partial_\eta M_{\gamma_\varepsilon}(\eta)\|
&\le4\delta_{\varepsilon,D}
+16\frac{\delta_{\varepsilon,D}^2}{K_\varepsilon}
+4\delta_{\varepsilon,D}\\
&\le12\delta_{\varepsilon,D},
\end{aligned}
\]
since \(K_\varepsilon\ge1\) and \(\delta_{\varepsilon,D}\le1/4\).
This proves \eqref{eq:reduced-matrix-error-eta}.
\end{proof}

\subsection{Stability change and Hopf bifurcation}
\label{subsec:prey-hopf}

We now use \eqref{eq:spectral-equivalence} to prove the loss of stability through
a conjugate eigenvalue crossing and establish Hopf bifurcation at the crossing parameter.
Recall that \(\eta_*(\gamma)\), defined in \eqref{eq:matrix-threshold},
is the value of \(\eta\) at which the target matrix \(M_\gamma(\eta)\)
loses stability. The next lemma locates the crossing of the full
linearization near \(\eta_*(\gamma_\varepsilon)\).

\Needspace{5\baselineskip}
\begin{lemma}[Loss of stability through a conjugate eigenvalue pair]
\label{lem:critical-spectrum}
There exist constants \(\rho_*,\delta_{\mathrm{adm}}>0\),
independent of \(\varepsilon,D\), such that, for every pair
\((\varepsilon,D)\) satisfying
\[
0<\varepsilon<\varepsilon_{\mathrm{red}},\qquad
D>0,\qquad
\delta_{\varepsilon,D}\le\delta_{\mathrm{adm}},
\]
the following assertions hold.
\begin{enumerate}
\item[(i)]
The interval
\[
I_\varepsilon^{\mathrm{root}}
:=[\eta_*(\gamma_\varepsilon)-\rho_*,
   \eta_*(\gamma_\varepsilon)+\rho_*]
\]
is compactly contained in \(I\).
It contains a unique parameter \(\eta_{\varepsilon,D}\)
at which \(\cL_{\varepsilon,D,\eta_{\varepsilon,D}}\)
has a pair of algebraically simple eigenvalues
\[
\lambda_{\varepsilon,D}^{\pm}
=\pm i\omega_{\varepsilon,D},
\qquad \omega_{\varepsilon,D}>0.
\]

\item[(ii)]
At \(\eta=\eta_{\varepsilon,D}\),
\[
\sigma(\cL_{\varepsilon,D,\eta_{\varepsilon,D}})
\setminus
\{\lambda_{\varepsilon,D}^{+},\lambda_{\varepsilon,D}^{-}\}
\subset\{z\in\C:\operatorname{Re}z<0\}.
\]

\item[(iii)]
The local eigenvalue branch
\(\lambda_{\varepsilon,D}^{+}(\eta)\) satisfies
\[
\left.
\partial_\eta\operatorname{Re}\lambda_{\varepsilon,D}^{+}(\eta)
\right|_{\eta=\eta_{\varepsilon,D}}>0.
\]
\end{enumerate}
\end{lemma}

\begin{proof}
We assume
\(\delta_{\varepsilon,D}\le\min\{1/4,\pi^2/(2R)\}\) and further restrict
this bound below, always independently of \(\varepsilon,D\).

(i) 
Applying the implicit-function theorem to the simple eigenvalues in
\eqref{eq:target-critical-spectrum} gives jointly \(C^1\) eigenvalue branches
of \(M_\gamma(\eta)\),
\(z_{0,\gamma}\), \(z_{+,\gamma}\), and
\(z_{-,\gamma}=\overline{z_{+,\gamma}}\) near the crossing curve.
Compactness of \([0,\gamma_0]\) gives a uniform \(\rho_*>0\) such that
\([\eta_*(\gamma)-\rho_*,\eta_*(\gamma)+\rho_*]\Subset I\).
On these intervals the roots are simple and uniformly separated.
By \eqref{eq:target-critical-spectrum} and
\eqref{eq:target-crossing-direction}, we may reduce \(\rho_*\) so that,
for some \(c>0\),
\begin{equation}
\Im z_{+,\gamma}(\eta)\ge c,\qquad
\partial_\eta\operatorname{Re} z_{+,\gamma}(\eta)\ge c,\qquad
z_{0,\gamma}(\eta)\le-c.
\label{eq:model-branch-signs}
\end{equation}
By \eqref{eq:spectral-window-radius}, we have
\(|z_{j,\gamma}(\eta)|<R/2\). Choose a uniform \(r>0\) such that the disks
\begin{equation}
\{z\in\C:|z-z_{j,\gamma}(\eta)|\le r\},\qquad j=0,+,-,
\label{eq:reference-root-disks}
\end{equation}
are pairwise disjoint and contained in \(\{|z|<R\}\).
For \(\gamma=\gamma_\varepsilon\), write
\(\mathfrak F_0(z,\eta):=zI_3-M_{\gamma_\varepsilon}(\eta)\).
On the compact set
\[
\Omega_\varepsilon
:=\{(z,\eta)\in\C\times I_\varepsilon^{\mathrm{root}}:|z|\le R\},
\]
the determinant is a polynomial in the matrix entries, so
\eqref{eq:reduced-matrix-approximation} gives
\begin{equation}
\|\det\mathfrak F_{\varepsilon,D}-\det\mathfrak F_0\|_{C^1(\Omega_\varepsilon)}
\le C\delta_{\varepsilon,D}.
\label{eq:determinant-C1-bound}
\end{equation}
Here the norm includes the first derivatives in \(z,\eta\).
On \(|z|\le R\) outside the interiors of the disks
\eqref{eq:reference-root-disks}, compactness gives a uniform positive lower
bound for \(|\det(zI_3-M_\gamma(\eta))|\).
Consequently, for small \(\delta_{\varepsilon,D}\),
\eqref{eq:determinant-C1-bound} implies there, with \(\gamma=\gamma_\varepsilon\),
\begin{equation}
|\det\mathfrak F_{\varepsilon,D}-\det\mathfrak F_0|
<|\det\mathfrak F_0|.
\label{eq:determinant-rouche-bound}
\end{equation}
Both determinants in \eqref{eq:determinant-rouche-bound} are analytic in \(z\).
Rouch\'e's theorem on each boundary circle therefore gives exactly one zero
of \(\det\mathfrak F_{\varepsilon,D}\) in each disk, counted with multiplicity.
These zeros are simple; denote them by \(z_{j,\varepsilon,D}(\eta)\).
Inequality \eqref{eq:determinant-rouche-bound} also excludes zeros in the
rest of \(\{|z|\le R\}\).
For \(\operatorname{Re} z\ge-R\), \(|z|\ge R\), equations
\eqref{eq:spectral-window-radius} and \eqref{eq:reduced-matrix-error-zero} give
\begin{equation}
\begin{aligned}
\|\mathfrak F_0^{-1}\|
&\le\frac1{|z|-\|M_{\gamma_\varepsilon}(\eta)\|}\le\frac2{|z|},\qquad
\|\mathfrak F_0^{-1}(\mathfrak F_{\varepsilon,D}-\mathfrak F_0)\|
\le\frac{8\delta_{\varepsilon,D}}{|z|}\le\tfrac12,
\end{aligned}
\label{eq:large-z-model-resolvent}
\end{equation}
provided \(\delta_{\varepsilon,D}\le R/16\). The factorization
\begin{equation}
\mathfrak F_{\varepsilon,D}
=\mathfrak F_0\left[I_3+\mathfrak F_0^{-1}
(\mathfrak F_{\varepsilon,D}-\mathfrak F_0)\right]
\label{eq:reduced-invertibility-factorization}
\end{equation}
has two invertible factors by \eqref{eq:large-z-model-resolvent}.
Thus the three simple roots account for all zeros in \(\operatorname{Re} z\ge-R\).

Simplicity of the zeros and Lemma~\ref{lem:reduced-matrix-approximation}
allow the implicit-function theorem to give \(C^1\) branches satisfying
\begin{equation}
\det\mathfrak F_{\varepsilon,D}
\bigl(z_{j,\varepsilon,D}(\eta),\eta\bigr)=0,
\qquad j=0,+,-.
\label{eq:reduced-zero-branches}
\end{equation}
Uniqueness in the disks makes these local branches agree on overlaps.
To estimate their displacement, factor the reference polynomial as
\[
\det\mathfrak F_0(z,\eta)
=\prod_{k\in\{0,+,-\}}(z-z_{k,\gamma_\varepsilon}(\eta)).
\]
In the \(j\)-th disk, the other two factors have a uniform positive lower
bound. Hence \eqref{eq:determinant-C1-bound} and
\eqref{eq:reduced-zero-branches} yield
\begin{equation}
|z_{j,\varepsilon,D}-z_{j,\gamma_\varepsilon}|
\le C|\det\mathfrak F_0(z_{j,\varepsilon,D},\eta)|
\le C\delta_{\varepsilon,D}.
\label{eq:root-displacement}
\end{equation}
At a reference root, \(|\partial_z\det\mathfrak F_0|\) is uniformly
positive, since it is the product of the two root separations.
By \eqref{eq:determinant-C1-bound} and \eqref{eq:root-displacement},
\(|\partial_z\det\mathfrak F_{\varepsilon,D}|\) is also uniformly
bounded away from zero at the perturbed roots. Differentiating
\eqref{eq:reduced-zero-branches} gives
\begin{equation}
z_{j,\varepsilon,D}'(\eta)
=-\left.
\frac{\partial_\eta\det\mathfrak F_{\varepsilon,D}(z,\eta)}
{\partial_z\det\mathfrak F_{\varepsilon,D}(z,\eta)}
\right|_{z=z_{j,\varepsilon,D}(\eta)}.
\label{eq:reduced-root-derivative}
\end{equation}
The reference roots satisfy \eqref{eq:reduced-root-derivative} with
\(\mathfrak F_{\varepsilon,D}\) replaced by \(\mathfrak F_0\).
Its determinant is polynomial in \(z,\eta\), with uniformly bounded
derivatives on the region under consideration. Comparing the two quotients
using \eqref{eq:determinant-C1-bound} and \eqref{eq:root-displacement} gives
\begin{equation}
\|z_{j,\varepsilon,D}-z_{j,\gamma_\varepsilon}\|_
{C^1(I_\varepsilon^{\mathrm{root}})}
\le C_{\mathrm{rt}}\delta_{\varepsilon,D},\qquad j=0,+,-,
\label{eq:root-C1}
\end{equation}
where \(C_{\mathrm{rt}}>0\) is independent of \(\varepsilon,D\).
The coefficients are real, so uniqueness in the disks also gives
\begin{equation}
z_{0,\varepsilon,D}(\eta)\in\mathbb R
\quad\text{and}\quad
z_{-,\varepsilon,D}(\eta)=\overline{z_{+,\varepsilon,D}(\eta)}.
\label{eq:reduced-zero-conjugacy}
\end{equation}
Since \(\operatorname{Re} z_{+,\gamma}(\eta_*(\gamma))=0\), integration of the second
inequality in \eqref{eq:model-branch-signs} gives
\begin{equation}
\pm\operatorname{Re} z_{+,\gamma}(\eta_*(\gamma)\pm\rho_*)
=\int_0^{\rho_*}\partial_\eta\operatorname{Re} z_{+,\gamma}
(\eta_*(\gamma)\pm s)\,\dd s\ge c\rho_*.
\label{eq:model-endpoint-signs}
\end{equation}
Choose \(C_{\mathrm{rt}}\delta_{\varepsilon,D}
<c\min\{1,\rho_*\}/2\). Equations \eqref{eq:model-branch-signs} and
\eqref{eq:root-C1} then imply
\begin{equation}
\Im z_{+,\varepsilon,D}(\eta)>0,\qquad
\partial_\eta\operatorname{Re} z_{+,\varepsilon,D}(\eta)>0,\qquad
z_{0,\varepsilon,D}(\eta)<0,
\label{eq:perturbed-branch-signs}
\end{equation}
while \eqref{eq:model-endpoint-signs} gives
\begin{equation}
\operatorname{Re} z_{+,\varepsilon,D}
\bigl(\eta_*(\gamma_\varepsilon)-\rho_*\bigr)<0
<\operatorname{Re} z_{+,\varepsilon,D}
\bigl(\eta_*(\gamma_\varepsilon)+\rho_*\bigr).
\label{eq:perturbed-endpoint-signs}
\end{equation}
Thus \(\operatorname{Re} z_{+,\varepsilon,D}\) is strictly increasing, with opposite
signs at the endpoints. It has exactly one zero
\(\eta_{\varepsilon,D}\) in the interior of \(I_\varepsilon^{\mathrm{root}}\).
At this point, \eqref{eq:root-C1} gives
\(|\operatorname{Re} z_{+,\gamma_\varepsilon}(\eta_{\varepsilon,D})|
\le C_{\mathrm{rt}}\delta_{\varepsilon,D}\).
The mean-value theorem and \eqref{eq:model-branch-signs} therefore yield
\begin{equation}
|\eta_{\varepsilon,D}-\eta_*(\gamma_\varepsilon)|
\le\frac{C_{\mathrm{rt}}}{c}\delta_{\varepsilon,D}.
\label{eq:true-threshold}
\end{equation}

Choose \(\delta_{\mathrm{adm}}>0\) to meet all the restrictions above;
they depend only on the fixed matrix family and its chosen neighborhoods.
The operator \(\cL_{\varepsilon,D,\eta}\) has compact resolvent.
Put \(\lambda_{\varepsilon,D}^j(\eta):=
\Lambda_{1,\varepsilon}z_{j,\varepsilon,D}(\eta)\).
By Lemma~\ref{lem:three-mode-reduction}, the three simple zeros located using
\eqref{eq:determinant-rouche-bound}--\eqref{eq:reduced-invertibility-factorization}
give, for every \(\eta\in I_\varepsilon^{\mathrm{root}}\),
\begin{equation}
\sigma(\cL_{\varepsilon,D,\eta})
\cap\{\lambda\in\C:\operatorname{Re}\lambda\ge-R\Lambda_{1,\varepsilon}\}
=\{\lambda_{\varepsilon,D}^j(\eta):j=0,+,-\},
\label{eq:critical-spectrum-localization}
\end{equation}
with all three eigenvalues algebraically simple.
Define
\(\omega_{\varepsilon,D}:=
\Lambda_{1,\varepsilon}\Im z_{+,\varepsilon,D}(\eta_{\varepsilon,D})\).
By \eqref{eq:reduced-zero-conjugacy} and the first inequality in
\eqref{eq:perturbed-branch-signs}, we have
\begin{equation}
\lambda_{\varepsilon,D}^{\pm}(\eta_{\varepsilon,D})
=\pm i\omega_{\varepsilon,D},
\qquad \omega_{\varepsilon,D}>0.
\label{eq:true-pair}
\end{equation}
By \eqref{eq:critical-spectrum-localization} and
\eqref{eq:perturbed-branch-signs}, a purely imaginary eigenvalue can occur
only at the unique zero of \(\operatorname{Re}z_{+,\varepsilon,D}\).
This proves (i).

{(ii)}
The third inequality in \eqref{eq:perturbed-branch-signs} gives
\(\lambda_{\varepsilon,D}^0<0\).
By \eqref{eq:critical-spectrum-localization}, all eigenvalues other than
\(\lambda_{\varepsilon,D}^{0},\lambda_{\varepsilon,D}^{+},
\lambda_{\varepsilon,D}^{-}\) satisfy
\(\operatorname{Re}\lambda<-R\Lambda_{1,\varepsilon}<0\).
In particular, (ii) holds at \(\eta=\eta_{\varepsilon,D}\).

{(iii) }
Since \(\Lambda_{1,\varepsilon}>0\) is independent of \(\eta\), the second inequality
in \eqref{eq:perturbed-branch-signs} gives
\begin{equation}
\left.\partial_\eta\operatorname{Re}\lambda_{\varepsilon,D}^{+}(\eta)
\right|_{\eta=\eta_{\varepsilon,D}}
=\Lambda_{1,\varepsilon}
\left.\partial_\eta\operatorname{Re} z_{+,\varepsilon,D}(\eta)
\right|_{\eta=\eta_{\varepsilon,D}}>0.
\label{eq:true-crossing-positive}
\end{equation}
This proves (iii).
\end{proof}

\begin{proof}[Proof of Theorem~\ref{thm:main}]
Take the coefficient families and unique equilibria from
Lemma~\ref{lem:prey-coefficients}. By \eqref{eq:scale-separation} and
\eqref{eq:profile-bound}, we have
\[
\frac{\Lambda_{1,\varepsilon} K_\varepsilon^2}{\Lambda_{2,\varepsilon}}
\le\frac{8C_1C_{\mathrm{pr}}^2}{\nu_2-\nu_1}
\varepsilon^{-5}e^{-\kappa_1/\varepsilon}\longrightarrow0.
\]
Choose \(0<\varepsilon_0\le\varepsilon_{\mathrm{red}}\), below the thresholds
in Lemmas~\ref{lem:half-local-spectrum}--\ref{lem:prey-coefficients}, such that
\begin{equation}
\frac{\Lambda_{1,\varepsilon} K_\varepsilon^2}{\Lambda_{2,\varepsilon}}
<\frac{\delta_{\mathrm{adm}}}{2},\qquad 0<\varepsilon<\varepsilon_0.
\label{eq:first-smallness-threshold}
\end{equation}
Set
\begin{equation}
D_0(\varepsilon)
:=\frac{2\Lambda_{1,\varepsilon} K_\varepsilon^2}{\delta_{\mathrm{adm}}}.
\label{eq:predator-diffusion-threshold}
\end{equation}
For \(D>D_0(\varepsilon)\), we obtain
\[
\delta_{\varepsilon,D}
=\frac{\Lambda_{1,\varepsilon} K_\varepsilon^2}{\Lambda_{2,\varepsilon}}
+\frac{\Lambda_{1,\varepsilon} K_\varepsilon^2}{D}
<\frac{\delta_{\mathrm{adm}}}{2}
+\frac{\delta_{\mathrm{adm}}}{2}
=\delta_{\mathrm{adm}}.
\]
Hence Lemma~\ref{lem:critical-spectrum} gives the crossing
\(\eta_{\varepsilon,D}\). Moreover,
\begin{equation}
0<D_0(\varepsilon)
\le\frac{4C_1C_{\mathrm{pr}}^2}{\delta_{\mathrm{adm}}}
\varepsilon^{-4}e^{-\kappa_1/\varepsilon}\longrightarrow0.
\label{eq:D0-decay}
\end{equation}

The spectral bound
\(\spb(L):=\sup_{\lambda\in\sigma(L)}\operatorname{Re}\lambda\)
is the supremum of the real parts of the spectral values of \(L\).
Choose \(\rho_{\mathrm H}>0\) with
\[
[\eta_{\varepsilon,D}-\rho_{\mathrm H},
\eta_{\varepsilon,D}+\rho_{\mathrm H}]
\Subset I_\varepsilon^{\mathrm{root}}.
\]
Equations \eqref{eq:critical-spectrum-localization},
\eqref{eq:perturbed-branch-signs}, and \eqref{eq:true-pair} then give
\begin{subequations}
\label{eq:full-stability-switch}
\begin{align}
&\spb(\cL_{\varepsilon,D,\eta})<0
\quad\text{if }\eta_{\varepsilon,D}-\rho_{\mathrm H}
<\eta<\eta_{\varepsilon,D},
\label{eq:full-stability-before}\\
&\spb(\cL_{\varepsilon,D,\eta})>0
\quad\text{if }\eta_{\varepsilon,D}<\eta
<\eta_{\varepsilon,D}+\rho_{\mathrm H}.
\label{eq:full-stability-after}
\end{align}
\end{subequations}

To complete the proof of Theorem~\ref{thm:main}, we work on
\[
X:=C([0,1];\R)^2,\qquad
\operatorname{Dom}(\cL_{\varepsilon,D,\eta})
=C_N^2([0,1];\R)^2,
\]
where
\[
C_N^2([0,1];\R)
:=\{w\in C^2([0,1];\R):w'(0)=w'(1)=0\}.
\]
The negative diagonal diffusion operator is sectorial, and the reaction
part is bounded on \(X\). By
\cite[Theorems~1.3.2 and~1.3.4]{Henry1981},
\(\cL_{\varepsilon,D,\eta}\) generates an analytic semigroup.
The same bounded-perturbation argument gives compact resolvent on \(X\).
By elliptic regularity for generalized eigenvectors,
the realizations of \(\cL_{\varepsilon,D,\eta}\) on \(C([0,1])^2\)
and \(L^2((0,1))^2\) have the same spectrum, including algebraic multiplicities.

After translating the equilibrium to zero, the remainder is
\begin{equation}
\mathcal N_\eta(\xi,\psi):=
\begin{pmatrix}
-a_\varepsilon\xi^2-b_\varepsilon\xi\psi\\
c_{\varepsilon,\eta}\xi\psi-e_\varepsilon\psi^2
\end{pmatrix}.
\label{eq:Hopf-quadratic-remainder}
\end{equation}
Bounded pointwise multiplication on \(C([0,1])\) and the smooth coefficients give
\[
\bigl[(\eta,w)\mapsto\mathcal N_\eta(w)\bigr]\in C^\infty(I\times X,X),\qquad
\mathcal N_\eta(0)=0,\qquad D_w\mathcal N_\eta(0)=0.
\]
To fix the linear part at the crossing, put
\begin{equation}
\mu:=\eta-\eta_{\varepsilon,D},\qquad
s:=\omega_{\varepsilon,D}t,\qquad w:=(\xi,\psi).
\label{eq:Hopf-parameter-time-rescaling}
\end{equation}
The translated equation becomes
\[
\frac{\dd w}{\dd s}
=\frac1{\omega_{\varepsilon,D}}
\cL_{\varepsilon,D,\eta_{\varepsilon,D}}w+\mathcal G(\mu,w),
\]
where
\[
\mathcal G(\mu,w):=\frac1{\omega_{\varepsilon,D}}
\left[
(\cL_{\varepsilon,D,\eta_{\varepsilon,D}+\mu}
-\cL_{\varepsilon,D,\eta_{\varepsilon,D}})w
+\mathcal N_{\eta_{\varepsilon,D}+\mu}(w)
\right].
\]
The operator difference is bounded on \(X\) and smooth in \(\mu\); hence
\[
\mathcal G\in C^\infty,\qquad
\mathcal G(\mu,0)=0,\qquad D_w\mathcal G(0,0)=0.
\]
By Lemma~\ref{lem:critical-spectrum}, we also have
\[
\sigma\left(
\omega_{\varepsilon,D}^{-1}
\cL_{\varepsilon,D,\eta_{\varepsilon,D}}\right)\cap i\R
=\{-i,i\},
\qquad
\left.\partial_\mu\operatorname{Re}
\frac{\lambda_{\varepsilon,D}^+(\eta_{\varepsilon,D}+\mu)}
{\omega_{\varepsilon,D}}\right|_{\mu=0}>0,
\]
with both imaginary eigenvalues algebraically simple.
Thus \cite[Theorem~1.11]{CrandallRabinowitzHopf} applies with
fractional-domain exponent \(\alpha=0\), giving nonstationary time-periodic
solutions near the equilibrium. They remain strictly positive because
\[
\min\{\min_{x\in[0,1]}\phi_{1,\varepsilon}(x),V_\varepsilon\}>0.
\]
Parabolic regularity and periodicity make these solutions classical.

Restore the family label \(\tau\in\{+,-\}\), writing the crossing as
\(\eta_{\varepsilon,D}^\tau\), and normalize \(I=[\eta_{\mathrm L},\eta_{\mathrm R}]\):
\begin{equation}
\Delta_I:=\eta_{\mathrm R}-\eta_{\mathrm L}>0,\qquad
\widehat\eta:=\frac{\eta-\eta_{\mathrm L}}{\Delta_I}\in[0,1].
\label{eq:parameter-normalization}
\end{equation}
For every parameter-dependent coefficient, set
\(\widehat f_{\widehat\eta}:=f_{\eta_{\mathrm L}+\Delta_I\widehat\eta}\).
The normalized crossing and its direction are
\[
\eta_{\mathrm H}^{\tau,\varepsilon,D}
:=\frac{\eta_{\varepsilon,D}^\tau-\eta_{\mathrm L}}{\Delta_I}\in(0,1),
\qquad
\partial_{\widehat\eta}\operatorname{Re}\lambda^+
=\Delta_I\partial_\eta\operatorname{Re}\lambda^+>0
\quad\text{at the crossing}.
\]
Renaming \(\widehat\eta\) as \(\eta\) proves the theorem for both families.
\end{proof}

\section{Proof for Spatially Concentrated Predator}
\label{sec:concentrated-predator}

In this section, we prove Theorem~\ref{thm:mirror-main} by exchanging the two components of
the adjoint of the linearization \eqref{eq:target-linearization}.
Set \(D=\varepsilon^2\).
Here \(\eta\) again ranges over the interval \(I\) in
\eqref{eq:coupling-parameter-interval}, with normalization to \([0,1]\)
postponed to the end. We retain the scalar spectral data of
Section~\ref{sec:preparations} and the coupling functions of
Lemma~\ref{lem:positive-profiles}. We reuse the coefficient notation from
Section~\ref{sec:concentrated-prey} for the new family.

\Needspace{5\baselineskip}
\begin{lemma}[Coefficient family with spatially concentrated predator]
\label{lem:mirror-coefficients}
For all sufficiently small \(\varepsilon\), there exist \(U_\varepsilon>0\)
and smooth coefficients, all independent of \(d\), such that
\[
r_\varepsilon,a_\varepsilon,b_\varepsilon,c_{\varepsilon,\eta},e_\varepsilon>0
\qquad m_{\varepsilon,\eta}<0
\quad\text{on }[0,1]\times I.
\]
For every \(d>0\) and \(\eta\in I\), the unique positive equilibrium is
\begin{equation}
(u_{\varepsilon,\eta}^*,v_{\varepsilon,\eta}^*)
=(U_\varepsilon,\phi_{1,\varepsilon}).
\label{eq:mirror-equilibrium}
\end{equation}
\end{lemma}

\begin{proof}
We prescribe \((u_*,v_*)=(U_\varepsilon,\phi_{1,\varepsilon})\). The prey
reaction bracket must vanish, while the predator reaction term must balance
\(\varepsilon^2\phi_{1,\varepsilon}''\). By the eigenfunction equation, the
balance identities are
\begin{equation}
r_\varepsilon-a_\varepsilon U_\varepsilon-b_\varepsilon\phi_{1,\varepsilon}=0,
\qquad
m_{\varepsilon,\eta}+c_{\varepsilon,\eta}U_\varepsilon
-e_\varepsilon\phi_{1,\varepsilon}=\lambda_{1,\varepsilon}-W.
\label{eq:mirror-balance-identities}
\end{equation}
Choose \(U_\varepsilon\) so large that
\begin{equation}
U_\varepsilon>1+
\max_{(x,\eta)\in[0,1]\times I}
\frac{\phi_{1,\varepsilon}(x)
\bigl(\lambda_{1,\varepsilon}-W(x)+\Lambda_{1,\varepsilon}/2\bigr)_+}
{\Lambda_{1,\varepsilon} k_{c,\varepsilon,\eta}(x)}.
\label{eq:mirror-amplitude-choice}
\end{equation}
Such a choice is possible by compactness and strict positivity of the denominator. Fix \(\gamma_\star\in(0,\gamma_0)\).
Define the six coefficients by
\begin{subequations}
\label{eq:mirror-coefficients}
\begin{align}
a_\varepsilon(x)&:=\frac{\gamma_\star \Lambda_{1,\varepsilon}}{U_\varepsilon},
\qquad
b_\varepsilon(x):=\frac{\Lambda_{1,\varepsilon} k_{b,\varepsilon}(x)}
{U_\varepsilon},
\label{eq:mirror-coefficients-ab}\\
e_\varepsilon(x)&:=\frac{\Lambda_{1,\varepsilon}}{2\phi_{1,\varepsilon}(x)},
\qquad
c_{\varepsilon,\eta}(x):=\frac{\Lambda_{1,\varepsilon}
k_{c,\varepsilon,\eta}(x)}{\phi_{1,\varepsilon}(x)},
\label{eq:mirror-coefficients-ec}\\
r_\varepsilon(x)
&:=\gamma_\star \Lambda_{1,\varepsilon}
+\frac{\Lambda_{1,\varepsilon}}{U_\varepsilon}
k_{b,\varepsilon}(x)\phi_{1,\varepsilon}(x),
\label{eq:mirror-coefficients-r}\\
m_{\varepsilon,\eta}(x)
&:=\lambda_{1,\varepsilon}-W(x)
+\frac{\Lambda_{1,\varepsilon}}{2}
-\frac{\Lambda_{1,\varepsilon} U_\varepsilon}{\phi_{1,\varepsilon}(x)}
k_{c,\varepsilon,\eta}(x).
\label{eq:mirror-coefficients-m}
\end{align}
\end{subequations}
Formulas \eqref{eq:mirror-coefficients-ab}--\eqref{eq:mirror-coefficients-ec}
prescribe the four entries of the linearized operator.
Formulas \eqref{eq:mirror-coefficients-r}--\eqref{eq:mirror-coefficients-m}
solve the two identities in \eqref{eq:mirror-balance-identities}, respectively.
By \eqref{eq:mirror-amplitude-choice}, we have
\[
\frac{\Lambda_{1,\varepsilon} U_\varepsilon}{\phi_{1,\varepsilon}(x)}
k_{c,\varepsilon,\eta}(x)
>
\bigl(\lambda_{1,\varepsilon}-W(x)+\Lambda_{1,\varepsilon}/2\bigr)_+
+\frac{\Lambda_{1,\varepsilon} k_{c,\varepsilon,\eta}(x)}
{\phi_{1,\varepsilon}(x)}.
\]
Consequently, the coefficients have the signs
\begin{subequations}
\label{eq:mirror-signs}
\begin{align}
&r_\varepsilon(x)>0,
\label{eq:mirror-sign-r}\\
&m_{\varepsilon,\eta}(x)
< -\frac{\Lambda_{1,\varepsilon} k_{c,\varepsilon,\eta}(x)}
{\phi_{1,\varepsilon}(x)}<0.
\label{eq:mirror-sign-m}
\end{align}
\end{subequations}
The inequalities \eqref{eq:mirror-sign-r} and \eqref{eq:mirror-sign-m}
hold for \((x,\eta)\in[0,1]\times I\).
Substituting \eqref{eq:mirror-coefficients} into
\eqref{eq:mirror-balance-identities} shows that
\eqref{eq:mirror-equilibrium} is an equilibrium. It is unique by
\cite[Theorem~4.1]{LopezGomezPardo1998}.
Direct differentiation gives
\begin{equation}
\cL^{\mathrm{mir}}_{\varepsilon,d,\eta}
=
\begin{pmatrix}
d\partial_{xx}-\gamma_\star \Lambda_{1,\varepsilon}\mathrm{Id}
&-\Lambda_{1,\varepsilon} k_{b,\varepsilon}\\
\Lambda_{1,\varepsilon} k_{c,\varepsilon,\eta}
&A_\varepsilon-\Lambda_{1,\varepsilon}\mathrm{Id}/2
\end{pmatrix}.\qedhere
\label{eq:mirror-linearization}
\end{equation}
\end{proof}

\begin{proof}[Proof of Theorem~\ref{thm:mirror-main}]
Use \(\rho_*,\delta_{\mathrm{adm}}\) from
Lemma~\ref{lem:critical-spectrum} and the coefficient family from
Lemma~\ref{lem:mirror-coefficients}. Choose \(\varepsilon_0\) as in
\eqref{eq:first-smallness-threshold}, small enough for
Lemma~\ref{lem:mirror-coefficients}. Set
\[
d_0(\varepsilon):=D_0(\varepsilon)
=\frac{2\Lambda_{1,\varepsilon} K_\varepsilon^2}{\delta_{\mathrm{adm}}}
\longrightarrow0
\]
by \eqref{eq:predator-diffusion-threshold} and \eqref{eq:D0-decay}.
For \(0<\varepsilon<\varepsilon_0\) and \(d>d_0(\varepsilon)\), we have
\[
\delta_{\varepsilon,d}
=\Lambda_{1,\varepsilon} K_\varepsilon^2
\left(\frac1{\Lambda_{2,\varepsilon}}+\frac1d\right)
<\frac{\delta_{\mathrm{adm}}}{2}
+\frac{\delta_{\mathrm{adm}}}{2}
=\delta_{\mathrm{adm}}.
\]

On \(L^2((0,1))^2\), with Neumann domain \(H_N^2((0,1))^2\), define
\begin{equation}
\cL^{\mathrm{aux}}_{\varepsilon,d,\eta}
:=
\begin{pmatrix}
A_\varepsilon-\Lambda_{1,\varepsilon}\mathrm{Id}/2
&-\Lambda_{1,\varepsilon} k_{b,\varepsilon}\\
\Lambda_{1,\varepsilon} k_{c,\varepsilon,\eta}
&d\partial_{xx}-\gamma_\star \Lambda_{1,\varepsilon}\mathrm{Id}
\end{pmatrix}.
\label{eq:mirror-auxiliary}
\end{equation}
For \(\mathscr S(\xi,\psi):=(\psi,\xi)\), the real coefficients and
self-adjoint diagonal blocks give
\begin{equation}
\mathscr S^{-1}=\mathscr S,\qquad
\cL^{\mathrm{mir}}_{\varepsilon,d,\eta}
=\mathscr S
\bigl(\cL^{\mathrm{aux}}_{\varepsilon,d,\eta}\bigr)^*
\mathscr S^{-1}.
\label{eq:adjoint-swap}
\end{equation}
The operators in \eqref{eq:mirror-linearization} and
\eqref{eq:mirror-auxiliary} have compact resolvent; their coefficients are real.
Hence
\begin{equation}
\begin{aligned}
\sigma\bigl(\cL^{\mathrm{mir}}_{\varepsilon,d,\eta}\bigr)
&=\left\{\overline{\lambda}:\lambda\in\sigma
\bigl(\cL^{\mathrm{aux}}_{\varepsilon,d,\eta}\bigr)\right\}=\sigma\bigl(\cL^{\mathrm{aux}}_{\varepsilon,d,\eta}\bigr),
\end{aligned}
\label{eq:mirror-spectral-equivalence}
\end{equation}
with algebraic multiplicities preserved.

The auxiliary operator is \eqref{eq:target-linearization} with
\(D=d\), \(\gamma_\varepsilon=\gamma_\star\). The estimates in
Section~\ref{subsec:prey-reduction} and the proof of
Lemma~\ref{lem:critical-spectrum} are uniform for
\(0\le\gamma\le\gamma_0\), so they apply on
\[
I_{\mathrm{mir}}^{\mathrm{root}}
:=[\eta_*(\gamma_\star)-\rho_*,
\eta_*(\gamma_\star)+\rho_*]\Subset I.
\]
\Needspace{9\baselineskip}
Together with \eqref{eq:mirror-spectral-equivalence}, they give a unique
parameter \(\eta_{\varepsilon,d}^{\mathrm{mir}}\) in the interior of
\(I_{\mathrm{mir}}^{\mathrm{root}}\) at which the algebraically simple
conjugate pair satisfies
\begin{equation}
\lambda_{\varepsilon,d}^{\mathrm{mir},\pm}
(\eta_{\varepsilon,d}^{\mathrm{mir}})
=\pm i\omega_{\varepsilon,d}^{\mathrm{mir}},
\qquad \omega_{\varepsilon,d}^{\mathrm{mir}}>0,
\label{eq:mirror-true-pair}
\end{equation}
and
\begin{equation}
\left.
\partial_\eta\operatorname{Re}\lambda_{\varepsilon,d}^{\mathrm{mir},+}(\eta)
\right|_{\eta=\eta_{\varepsilon,d}^{\mathrm{mir}}}>0.
\label{eq:mirror-true-crossing-positive}
\end{equation}
By \eqref{eq:critical-spectrum-localization} and the third inequality in
\eqref{eq:perturbed-branch-signs}, the remaining spectrum satisfies
\begin{equation}
\sigma\bigl(\cL^{\mathrm{mir}}_{\varepsilon,d,\eta}\bigr)
\setminus
\{\lambda_{\varepsilon,d}^{\mathrm{mir},+}(\eta),
\lambda_{\varepsilon,d}^{\mathrm{mir},-}(\eta)\}
\subset\{z\in\C:\operatorname{Re} z<0\},
\qquad \eta\in I_{\mathrm{mir}}^{\mathrm{root}}.
\label{eq:mirror-remaining-spectrum}
\end{equation}
The second inequality in \eqref{eq:perturbed-branch-signs} then gives,
for \(\eta\in I_{\mathrm{mir}}^{\mathrm{root}}\),
\[
\begin{cases}
\spb(\cL^{\mathrm{mir}}_{\varepsilon,d,\eta})<0,
&\eta<\eta_{\varepsilon,d}^{\mathrm{mir}},\\
\spb(\cL^{\mathrm{mir}}_{\varepsilon,d,\eta})>0,
&\eta>\eta_{\varepsilon,d}^{\mathrm{mir}}.
\end{cases}
\]

After translating \eqref{eq:mirror-equilibrium} to zero, denote the remainder
\eqref{eq:Hopf-quadratic-remainder}, with coefficients
\eqref{eq:mirror-coefficients}, by \(\mathcal N_\eta\). It satisfies
\[
\mathcal N_\eta(0)=0,\qquad D_w\mathcal N_\eta(0)=0,\qquad
\mathcal N_\eta\in C^\infty\bigl(C([0,1];\R)^2,C([0,1];\R)^2\bigr).
\]
The parameter smoothness, analytic-semigroup hypotheses, and spectral
realizations are as in the proof of Theorem~\ref{thm:main}.
Use the rescaling \eqref{eq:Hopf-parameter-time-rescaling} with
\(\eta_{\varepsilon,D},\omega_{\varepsilon,D}\) replaced by
\(\eta_{\varepsilon,d}^{\mathrm{mir}},\omega_{\varepsilon,d}^{\mathrm{mir}}\), respectively.
Then \eqref{eq:mirror-true-pair}--\eqref{eq:mirror-remaining-spectrum}
verify the hypotheses of
\cite[Theorem~1.11]{CrandallRabinowitzHopf}. This gives a local branch
of nonstationary classical solutions periodic in time.
These solutions are strictly positive, since they converge uniformly
to \((U_\varepsilon,\phi_{1,\varepsilon})\) and
\[
\min\{U_\varepsilon,\min_{x\in[0,1]}\phi_{1,\varepsilon}(x)\}>0.
\]
\Needspace{6\baselineskip}
Finally, \eqref{eq:parameter-normalization} gives
\[
\widehat\eta_{\mathrm H}
=\frac{\eta_{\varepsilon,d}^{\mathrm{mir}}-\eta_{\mathrm L}}{\Delta_I}
\in(0,1),\qquad
\partial_{\widehat\eta}\operatorname{Re}\lambda_{\varepsilon,d}^{\mathrm{mir},+}
=\Delta_I\partial_\eta\operatorname{Re}\lambda_{\varepsilon,d}^{\mathrm{mir},+},
\]
so the normalization preserves the crossing direction.
\end{proof}

\section*{Acknowledgments}

We acknowledge the assistance of AI tools in identifying references that helped us learn semiclassical analysis techniques used in the spectral reduction of the non-self-adjoint linearized operator. These tools also assisted with some calculations, English-language editing, and the presentation of some mathematical arguments during the preparation of this manuscript.

Y.~Lou was partially supported by the Interdisciplinary Disciplines
Breakthrough Plan of the Ministry of Education of China
(JYB2025XDXM902) and  NSF of China (12631017,12261160366). M.-Z.~Xin is partially supported by NSF of China (12601334), the China National
Postdoctoral Program for Innovative Talents (BX2026009), and the China
Postdoctoral Science Foundation (2026M793394).


\begin{thebibliography}{99}
\renewcommand{\baselinestretch}{1.00}\selectfont
\enlargethispage{2\baselineskip}

\bibitem{Amann2005}
H.~Amann,
Maximum principles and principal eigenvalues,
in \emph{Ten Mathematical Essays on Approximation in Analysis and Topology},
J.~Ferrera, J.~L\'opez-G\'omez and F.R.~Ruiz del Portal (eds.),
Elsevier, Amsterdam, 2005, pp.~1--60.

\bibitem{CantrellCosner1989}
R.S.~Cantrell and C.~Cosner,
Diffusive logistic equations with indefinite weights: population models in
disrupted environments,
\emph{Proc. Roy. Soc. Edinburgh Sect. A} \textbf{112} (1989) 293--318.

\bibitem{CantrellCosner19892003}
R.S.~Cantrell and C.~Cosner,
\emph{Spatial Ecology via Reaction-Diffusion Equations},
Wiley Series in Mathematical and Computational Biology,
John Wiley \& Sons, Ltd., Chichester, 2003.

\bibitem{CantrellCosner1991}
R.S.~Cantrell and C.~Cosner,
The effects of spatial heterogeneity in population dynamics,
\emph{J. Math. Biol.} \textbf{29} (1991) 315--338.



\bibitem{CrandallRabinowitzHopf}
M.G.~Crandall and P.H.~Rabinowitz,
The Hopf bifurcation theorem in infinite dimensions,
\emph{Arch. Rational Mech. Anal.} \textbf{67} (1977) 53--72.

\bibitem{DuHsu2004}
Y.~Du and S.-B.~Hsu,
A diffusive predator--prey model in heterogeneous environment,
\emph{J. Differential Equations} \textbf{203} (2004) 331--364.

\bibitem{DuLou1997}
Y.~Du and Y.~Lou,
Some uniqueness and exact multiplicity results for a predator--prey model,
\emph{Trans. Amer. Math. Soc.} \textbf{349} (1997) 2443--2475.

\bibitem{DuLou1998}
Y.~Du and Y.~Lou,
S-shaped global bifurcation curve and Hopf bifurcation of positive solutions to a
predator--prey model,
\emph{J. Differential Equations} \textbf{144} (1998) 390--440.

\bibitem{DuShi2007}
Y.~Du and J.~Shi,
Allee effect and bistability in a spatially heterogeneous predator--prey model,
\emph{Trans. Amer. Math. Soc.} \textbf{359} (2007) 4557--4593.

\bibitem{DuShi2006}
Y.~Du and J.~Shi,
A diffusive predator--prey model with a protection zone,
\emph{J. Differential Equations} \textbf{229} (2006) 63--91.

\bibitem{FournaisHelffer2010}
S.~Fournais and B.~Helffer,
\emph{Spectral Methods in Surface Superconductivity},
Progress in Nonlinear Differential Equations and Their Applications, vol.~77,
Birkh\"auser, Boston, 2010.

\bibitem{GohbergSigal1971}
I.C.~Gohberg and E.I.~Sigal,
An operator generalization of the logarithmic residue theorem and the theorem
of Rouch\'e,
\emph{Math. USSR-Sb.} \textbf{13} (1971) 603--625.

\bibitem{Helffer1988}
B.~Helffer,
\emph{Semi-Classical Analysis for the Schr\"odinger Operator and Applications},
Lecture Notes in Mathematics, vol.~1336, Springer, Berlin, 1988.

\bibitem{Henry1981}
D.~Henry,
\emph{Geometric Theory of Semilinear Parabolic Equations},
Lecture Notes in Mathematics, vol.~840, Springer, Berlin, 1981.

\bibitem{KutoLopezGomezMunoz2026}
K.~Kuto, J.~L\'opez-G\'omez and E.~Mu\~noz-Hern\'andez,
High multiplicity and global structure of coexistence states in a predator--prey
model with saturation,
\emph{J. Differential Equations} \textbf{462} (2026) 114116.

\bibitem{LamLouBook}
K.-Y.~Lam and Y.~Lou,
\emph{Introduction to Reaction-Diffusion Equations: Theory and Applications to
Spatial Ecology and Evolutionary Biology},
Lecture Notes on Mathematical Modelling in the Life Sciences,
Springer, Cham, 2022.

\bibitem{LiWuDong2019}
S.~Li, J.~Wu and Y.~Dong,
Uniqueness and stability of positive solutions for a diffusive predator--prey model in
heterogeneous environment,
\emph{Calc. Var. Partial Differential Equations} \textbf{58} (2019) 110.



%\bibitem{LopezGomezMunoz2024}
%J.~L\'opez-G\'omez and E.~Mu\~noz-Hern\'andez,
%A robust multiplicity result in a generalized diffusive predator--prey model,
%\emph{Adv. Differential Equations} \textbf{29} (2024) 437--476.

\bibitem{LopezGomezPardo1998}
J.~L\'opez-G\'omez and R.M.~Pardo,
Invertibility of linear noncooperative elliptic systems,
\emph{Nonlinear Anal.} \textbf{31} (1998) 687--699.

\bibitem{LopezNietoEtAl2025}
A.~L\'opez-Nieto, P.~Lappicy, N.~Vassena, H.~Stuke and J.-Y.~Dai,
Hybrid bifurcations: periodicity from eliminating a line of equilibria,
\emph{Math. Ann.} \textbf{391} (2025) 6373--6399.

\bibitem{Lotka1920}
A.J.~Lotka,
Analytical note on certain rhythmic relations in organic systems,
\emph{Proc. Natl. Acad. Sci. U.S.A.} \textbf{6} (1920) 410--415.

\bibitem{Lou2006}
Y.~Lou,
On the effects of migration and spatial heterogeneity on single and multiple
species,
\emph{J. Differential Equations} \textbf{223} (2006) 400--426.

\bibitem{OkuboLevin2001}
A.~Okubo and S.A.~Levin,
\emph{Diffusion and Ecological Problems: Modern Perspectives}, 2nd ed.,
Interdisciplinary Applied Mathematics, vol.~14,
Springer, New York, 2001.

\bibitem{PengWang2008}
R.~Peng and M.~Wang,
Uniqueness and stability of steady states for a predator--prey model in heterogeneous
environment,
\emph{Proc. Amer. Math. Soc.} \textbf{136} (2008) 859--865.

\bibitem{SegelJackson1972}
L.A.~Segel and J.L.~Jackson,
Dissipative structure: An explanation and an ecological example,
\emph{J. Theoret. Biol.} \textbf{37} (1972) 545--559.

\bibitem{ShuXinWangYu2026}
H.~Shu, S.~Xin, X.-S.~Wang and J.~Yu,
Joint impact of diffusion and advection on predator--prey dynamics with
delay and fear effect,
\emph{SIAM J. Appl. Math.} \textbf{86} (2026) 76--103.

\bibitem{Simon1983}
B.~Simon,
Semiclassical analysis of low lying eigenvalues. I. Non-degenerate minima:
asymptotic expansions,
\emph{Ann. Inst. H. Poincar\'e Sect. A (N.S.)} \textbf{38} (1983) 295--308.

\bibitem{Skellam1951}
J.G.~Skellam,
Random dispersal in theoretical populations,
\emph{Biometrika} \textbf{38} (1951) 196--218.

\bibitem{Volterra1926}
V.~Volterra,
Fluctuations in the abundance of a species considered mathematically,
\emph{Nature} \textbf{118} (1926) 558--560.

\bibitem{YiWeiShi2009}
Y.~Yi, J.~Wei and J.~Shi,
Bifurcation and spatiotemporal patterns in a homogeneous diffusive predator--prey system,
\emph{J. Differential Equations} \textbf{246} (2009), no.~5, 1944--1977.
\end{thebibliography}
\end{document}